%% file: main.tex
\documentclass{article}

\usepackage{amsmath,amssymb,amsfonts,amsthm}
\usepackage{stmaryrd}
\usepackage{mathtools}
\usepackage{sidecap}
\usepackage[margin=2cm]{geometry}
\usepackage{hyperref}
\usepackage[
  backend=biber,
  style=alphabetic,
]{biblatex}
\usepackage{pgfplots}
\pgfplotsset{compat=1.15}
\usepackage{mathrsfs}
\usetikzlibrary{arrows,shapes.geometric}
\definecolor{ffffff}{rgb}{1,1,1}

\usepackage{cleveref}

\sidecaptionvpos{figure}{m}

\newcommand{\ind}{\mathbf 1}
\newcommand{\E}{\mathbb E}
\newcommand{\Proba}{\mathbb P}
\newcommand{\R}{\mathbb R}
\newcommand{\N}{\mathbb N}
\newcommand{\Z}{\mathbb Z}
\newcommand{\Cov}{\mathrm{Cov}}
\newcommand{\Var}{\mathrm{Var}}
\newcommand{\diff}{\mathrm d}
\newcommand{\pivotal}{\text{ pivotal for }}
\newcommand{\occ}{\text{ occupied}}

\newcommand{\lra}{\leftrightarrow}

\newcommand{\nei}{\mathcal N}
\renewcommand{\le}{\leqslant}
\renewcommand{\ge}{\geqslant}

\numberwithin{equation}{section}

\newtheorem{theorem}{Theorem}[section]
\newtheorem{lemma}[theorem]{Lemma}
\newtheorem{prop}[theorem]{Proposition}
\newtheorem{coro}[theorem]{Corollary}
\newtheorem{definition}[theorem]{Definition}
\newtheorem{remark}[theorem]{Remark}

\AddToHook{env/lemma/begin}{\crefalias{theorem}{lemma}}
\AddToHook{env/prop/begin}{\crefalias{theorem}{prop}}
\AddToHook{env/coro/begin}{\crefalias{theorem}{coro}}
\AddToHook{env/definition/begin}{\crefalias{theorem}{definition}}
\AddToHook{env/remark/begin}{\crefalias{theorem}{remark}}

\title{Sharp phase transition for percolation with short-range dependencies}
\author{Olivier Henry, Peter Mörters}

\begin{document}
  \maketitle

  \input{sections/abstract.tex}

  \setcounter{tocdepth}{2}
  \tableofcontents
  
  \input{sections/references}
  \input{sections/site-bond}
  \input{sections/finite-weights}

  \printbibliography[heading=bibintoc]
\end{document}

%% file: sections/abstract.tex
\begin{abstract}
    We show sharpness of the phase transition for a nearest-neighbour percolation model on $\mathbb Z^d$, where vertices carry independent types and the percolation probability of edges depends on the type of the adjacent vertices. Our proof uses the OSSS inequality and adapts to our setup the method developed in Duminil-Copin et al.~\cite{drt17} for the random cluster model.  Additionally, we provide a more extensive study of the special case of combined Bernoulli bond and site percolation featuring a phase transition with two parameters.
\end{abstract}

%% file: sections/references.tex
\section{Introduction}

Percolation on an infinite, transitive graph $\mathscr G=(\mathscr V, \mathscr E)$ is any monotonically increasing  family $(\mu_p \colon p\in[0,1])$ of probability measures on the state space
$\Omega=\{0,1\}^{\mathscr E}$, which are invariant under graph automorphisms. We interpret $\omega(e)=1$ as the edge $e$ being open. For $x \in \mathscr V$ and nonnegative integer $n$, we let 
$\Lambda_n(x) = \{y \in \mathscr V \colon d(x,y) \le n\}$
be the ball around $x$ of radius~$n$ with respect to the graph metric~$d$, and $\partial\Lambda_{n+1}(x) = \Lambda_{n+1}(x) \backslash \Lambda_n(x)$.
We write $x \lra y$ if there is a path of open edges connecting $x$ and $y$, $x \lra V$ if there is a path of open edges connecting $x$ and an element of $V\subset\mathscr V$, 
and $x \lra \infty$ if there is a path of open edges connecting $x$ to points of arbitrarily large graph distance to $x$. Define
$$p_{\mathrm c} := \inf \{p \in [0, 1] \colon  \mu_p(x \lra \infty) > 0\}$$
and note that, by transitivity, this definition does not depend on the choice of $x \in \mathscr V$ and by monotonicity
$\mu_p(x \lra \infty) > 0$ for all $p>p_{\mathrm c}$. We say that the percolation model has a 
\textit{phase transition} if $p_{\mathrm c}\in(0,1)$ and the phase transition is \textit{sharp} if 
\begin{itemize}
  \item For $0\le p<p_{\mathrm c}$, there exists $c_p>0$ such that for any $n$ large enough, $\mu_p(0 \lra \partial \Lambda_n) \le \exp(-c_p n)$.
  \item There exists $c>0$ such that  $\mu_p(0 \lra \infty) \ge c(p-p_{\mathrm c})$ for all $1 \ge p \ge p_{\mathrm c}$.
\end{itemize}
The classical model is \emph{Bernoulli bond percolation} where $\mathscr G$ is the integer lattice $\mathbb Z^d$ with edges connecting $x,y\in\mathbb Z^d$ if and only if $\|x-y\|_2=1$ and $\mu_p$ the product measure with $\mu_p(\omega(e)=1)=p$ for any edge $e$.  For this model a sharp phase transition was first proved independently by Aizenman and Barsky~\cite{ab87} and Menshikov~\cite{men}, and recently beautiful new proofs were discovered by Duminil-Copin and Tassion~\cite{dt16}, Vanneuville~\cite{van25}, and 
Duminil-Copin et al.~\cite{drt17}. The latter proof 
applies to the random cluster model, a percolation model with dependent edges. The authors state that \emph{the techniques developed in this paper have potential to be
extended to a wide class of models including the Ashkin-Teller model, continuum percolation models such as Voronoi percolation and Boolean percolation, super-level sets of massive Gaussian free field, and random-cluster and Potts model with infinite range interactions.} Some of this research programme has already been carried out, see for example \cite{Dum19, Mu20, Beek21, Der21, Last23, Aoun24}.
The aim of the present note is to explore further how this proof idea can be used in models with dependent edges with a long-term view towards developping versatile methods for models with long-range dependence.
\medskip

In Section~\ref{sec:osss} we present the general setup of the proof of Duminil-Copin et al.~\cite{drt17} and in Section~\ref{sec:model} we introduce a model where 
edges are dependent if they share a vertex. 
In Sections~\ref{sec:site-bond} and~\ref{sec:fw} we show sharpness of the phase transition for instances of this model.

\subsection{The OSSS method}\label{sec:osss}

We briefly introduce the framework of the proof of a sharp phase transition 
as in~\cite{drt17}.
We start with the crucial inequality, the OSSS inequality, which allows us to bound the variance of the output of a randomized algorithm, or equivalently a decision tree.
\smallskip

A \emph{decision tree} is a family of functions $\phi_t : \{0, 1\}^{t-1} \to E$. For $\omega \in \{0, 1\}^E$, we define the edge $e_1(\omega) = \phi_1(\emptyset)$ then the edge $e_2(\omega) = \phi_2(\omega(e_1(\omega)))$, and more generally at time $t$ the edge $e_t(\omega) = \phi_t(\omega(e_1(\omega)), \ldots, \omega(e_{t-1}(\omega)))$\footnote{In \cite{drt17} the function $\phi$ is allowed to also take $e_1, \ldots, e_{t-1}$ as arguments, but these are determined by $\omega(e_1), \ldots, \omega(e_{t-1})$.}.
A tree usually computes the value of a function $f$, and we say that the tree stops at the minimal time $\tau(\omega)$ when the states of the edges $e_1, \dots, e_{\tau(\omega)}$ are sufficient to compute $f$, that is 
\begin{equation}
  \forall \psi \in \{0, 1\}^E \text{ such that } \forall 1 \le k \le \tau(\omega), \psi(e_k(\omega)) = \omega(e_k(\omega)) \text{ we have } f(\psi) = f(\omega).
\end{equation}
We do not need to define the tree after time $\tau$ and say that the edges $e_1, \dots, e_\tau$ are revealed.\smallskip

The (generalized) OSSS inequality bounds the variance of $f$ with the influence of the elements $e \in E$ over $f$ combined with the probability that they are revealed. It is proved in \cite[Theorem 1.1]{drt17}.

\begin{theorem}[Generalized OSSS inequality]
  Let $\mu$ be a monotonic\footnote{See \cite{gri06} for the definition and properties of monotonic measures. A measure is monotonic if and only if it has the FKG property.} measure on $\{0, 1\}^E$ where $E$ is finite. Let $T$ be a decision tree computing an increasing random variable $f$. Then
  \begin{equation}
      \Var(f) \le \sum\limits_{e \in E} \delta_e(f, T) \Cov(f, \omega(e))
  \end{equation}
  where
  \begin{equation}
    \delta_e(f, T) := \mu(e \text{ is revealed by } T)
  \end{equation}
  \label{thm:OSSS}
\end{theorem}


The following lemma is to be applied with $f_n = c\Proba(0 \lra \partial \Lambda_n)$ (with $c > 0$) and will immediately give the sharp phase transition. A proof can be found in \cite[Lemma 3.1]{drt17}.

\begin{lemma}
  Let $(f_n)_{n \ge 0}$ be a sequence of increasing differentiable functions from $[0, \beta_0]$ to $[0, M]$. Assume that this sequence converges to a function $f$ and that, for all $n \ge 1$, we have
  \begin{equation}
    f_n' \ge \frac n {\sum\limits_{k=0}^{n-1} f_k} f_n.
    \label{eq:3.1}
  \end{equation}
  Then there exists $\beta_1 \in [0, \beta_0]$ such that
  \begin{itemize}
      \item for $\beta < \beta_1$ there exists $c_\beta > 0$ such that for any $n$ large enough $f_n(\beta) \le \exp(- c_\beta n)$,
      \item for $\beta > \beta_1, f(\beta) \ge \beta - \beta_1$.
  \end{itemize}
  \label{lem:3.1}
\end{lemma}\pagebreak[3]


The first step in order to obtain the assumption of Lemma~\ref{lem:3.1} is the following intermediate result \cite[Lemma 3.2]{drt17}. We consider the functions $f_n = \ind_{0 \lra \partial \Lambda_n}$. To smoothen the discrepancies between different $\delta_e(T, f)$, Theorem~\ref{thm:OSSS} is applied to several decision trees $T_k$ for $1 \le k \le n$ before summing. Namely, $T_k$ is the decision tree associated with a graph walk in $\Lambda_n$ starting with all vertices in $\partial \Lambda_k$. It results in the sum of the $\delta_e(T_k, f)$ for a fixed $e$ being bounded by the denominator of the right-hand side of \eqref{eq:lem-4.6}.

\begin{lemma}
  Let $G = (V, E)$ be a finite graph containing $0$, and $\mu$ a monotonic measure on $\{0, 1\}^E$. Then, for any $n \in \N^*:=\{1,2,\ldots\}$, we have
  \begin{equation}
    \sum\limits_{e \in E} \Cov(\ind_{(0 \lra \partial \Lambda_n)}, \omega(e)) \ge \frac n {4 \max\limits_{x \in \Lambda_n} \sum\limits_{k=0}^{n-1} \mu(x \lra \partial \Lambda_k(x))} \mu(0 \lra \partial \Lambda_n) (1 - \mu(0 \lra \partial\Lambda_n)).
    \label{eq:lem-4.6}
  \end{equation}
  \label{lem:3.2}
\end{lemma}

{The max in the denominator of the right-hand side of \eqref{eq:lem-4.6} can be removed if we have translation invariance, resulting in the right-hand side of \eqref{eq:3.1}. In the case of product measures the left-hand side can be linked to $\frac \diff {\diff \beta} \mu(0 \lra \partial \Lambda_n)$ using Russo's formula (see \cite[Theorem 2.25]{gri99} for more details), effectively proving the sharp phase transition in that case.}
\smallskip

\begin{definition}
  Let $\Omega = \{0, 1\}^E$ with $E$ at most countable. For $\omega \in \Omega$ and $F, G \subset E$ such that $F \cap G = \emptyset$, we denote
  \begin{equation}
    \omega^F_G : e \in E \mapsto \begin{cases}
      1 & \text{ if } e \in F \\
      0 & \text{ if } e \in G \\
      \omega(e) & \text{ otherwise}
    \end{cases}
  \end{equation}
  If $F$ or $G$ is empty, we omit it. If $F$ or $G$ is a singleton $\{e\}$, we write $e$ instead of $\{e\}$\footnote{We always write $\omega(e)$ for the actual value of the $e$-component of $\omega$ in order to avoid confusion.}.
  Consider an increasing event $A \subset \Omega$. We define the event that $e \in E$ is pivotal for $A$ to be
  \begin{equation}
    \{\omega \in \Omega: (\omega_e \notin A) \wedge (\omega^e \in A)\}.
  \end{equation}
\end{definition}

\begin{theorem}[Russo's formula]
  Let $\Omega = \{0, 1\}^E$ with $E$ at most countable and, for $p \in [0, 1]^E$, let $\Proba_p$ be the product probability measure on $\Omega$ such that for all $e \in E$ we have $\Proba_p(\omega(e) = 1) = p_e$. Then, if $A$ is an increasing event and~$e\in E$,
  \begin{equation}
    \partial_{p_e} \Proba_p(A) = \Proba_p(e \pivotal A)
  \end{equation}
  \label{thm:Russo}
\end{theorem}

Theorem~\ref{thm:Russo} as stated is in fact an intermediate step in the proof of the formula in \cite[Theorem 2.25]{gri99}. When applying it  we need additional flexibility compared to the formulation of~\cite[Theorem 2.25]{gri99}.

\begin{remark}
  Combining Russo's formula with the chain rule we can differentiate $\Proba(A)$, when $A$ depends on finitely many components. Indeed, if $A$ depends on the components in $F \subset E$ with $F$ finite, and for all $e \in E$ the function $p_e : \beta \mapsto p_e(\beta)$ is of class $C^1$, then
  \begin{equation}
    \frac \diff {\diff \beta} \Proba_{p(\beta)}(A) = \sum\limits_{e \in F} p_e'(\beta) \Proba_{p(\beta)}(e \pivotal A)
  \end{equation}
  Russo's formula as needed for Bernoulli percolation, and as provided in \cite[Theorem 2.25]{gri99}, is the formula above for the special case $p_e(\beta) = \beta$.
\end{remark}

Now that we can express $\frac \diff {\diff p} \Proba_p(0 \lra \partial \Lambda_n)$, we need to compute the covariances in the left-hand side of \eqref{eq:lem-4.6}.

\begin{lemma}
  Let $E$ be a set and $\Proba$ a product probability measure on $\{0, 1\}^E$. Let $A$ be an increasing event and~$e\in E$. Then
  \begin{equation}
    \Cov(\omega(e), \ind_A) = \Var(\omega(e)) \Proba(e \pivotal A)
  \end{equation}
  \label{lem:cov}
\end{lemma}

\begin{proof}
  \begin{align*}
    \Cov(\omega(e), \ind_A) & = \Cov(\omega(e), \ind_{e \pivotal A, \omega(e) = 1} + \ind_A(\omega_e)) \\
                           & = \Cov(\omega(e), \omega(e) \ind_{e \pivotal A}) + 0 \\
                           & = \E((\omega(e) - \E(\omega(e))) \omega(e) \ind_{e \pivotal A}) \\
                           & = \E((\omega(e) - \E(\omega(e))) \omega(e)) \E(\ind_{e \pivotal A}) \\
                           & = \Var(\omega(e)) \Proba(e \pivotal A)
  \end{align*}
\end{proof}

The final steps in the case of product measures consist of (locally) bounding the remaining factors before using Lemma~\ref{lem:3.1}. 

\subsection{Percolation with short-range dependencies}\label{sec:model}

The model that fueled this research is the following. We consider the lattice $\Z^d$ with edges $E = \{\{x, y\} \subset \Z^d: \|x-y\|_1 = 1\}$ and the parameter $\beta \in \R_+$. Let $\nu_x \ge 0$ for $x \in \Z^d$ be iid weights and $\eta : \R^2 \to \R_+$ a nondecreasing function. Given the weights $\nu$, the states of the edges are conditionally  independent and an edge $e = \{x, y\}$ is open with probability $1 - \exp(- \beta \eta(\nu_x, \nu_y))$.
\smallskip

In this context, we denote $\Lambda_n(x) = \{y \in \Z^d: \|x-y\|_\infty \le n\}$ instead of using the graph distance.

\subsubsection{Site and bond percolation}

We start in Section~\ref{sec:site-bond} by exploring the very simple case where $\eta$ is the product or the minimum and the weights $(\nu_x \colon x\in\mathbb Z^d)$ are iid Bernoulli variables of parameter $q$, it is equivalent to independent Bernoulli bond percolation with parameter $p = 1 - \exp(-\beta)$ and site percolation with paramater $q$. We are therefore primarily interested in the behaviour of this model with fixed $q$, but also explore the phase transition in the square $[0, 1]^2$.\smallskip

We prove that there are open subcritical and supercritical phases separated by a critical curve, with some regularity properties (bounds on its slope). Moreover, we prove exponential decay in the subcritical phase and the inequality $\Proba_{\gamma(t)}(0 \lra \infty) \ge c(t - t_c)$ for some curves $\gamma$, depending on their slopes when crossing the critical curve.

\subsubsection{Percolation with finitely many weights}

This first example provided enough insight to study the case where the weights $\nu_x$ have finitely many possible values, and the proof of the sharpness of the potential phase transition we obtained holds for a more general class of models, as described in Section~\ref{sec:fw}.

%% file: sections/site-bond.tex
\section{Phase transition for  site and bond Bernoulli percolation}
\label{sec:site-bond}

Here we work on the case where $\eta$ is the product or the minimum and the weights $\nu_x$ are Bernoulli variables. The model is equivalent to independent site percolation with parameter $q \in [0, 1]$ and bond percolation with parameter $p \in [0, 1]$ at the same time.

\subsection{Using the OSSS inequality on site percolation}

We prove a counterpart of Lemma~\ref{lem:3.2} in the case of site percolation.

\begin{lemma}
  Let $G = (V, E)$ be a finite subgraph of $\Z^d$ containing $0$. For $n \in \N:=\{0,1,2,\ldots\}$ and $x \in V$ we denote $L_n(x) = \Lambda_n(x) \cap V$ and $\partial L_n(x) = \partial \Lambda_n(x) \cap V$. Let $\mu$ be an increasing measure on $\{0, 1\}^V$. For all $n \in \N^*$ we have
  \begin{equation}
    \sum\limits_{x \in V} \Cov(\ind_{0 \lra \partial L_n}, \omega(x)) \ge \frac n 2 \mu(0 \lra \partial L_n)(1 - \mu(0 \lra \partial L_n)) \min\limits_{y \in L_n} \frac {\mu(\omega(y) = 1)} {\sum\limits_{k=0}^{n-1} \mu(y \lra \partial L_k(y))}.
    \label{eq:4.6-site}
  \end{equation}
  \label{lem:4.6-site}
\end{lemma}

\begin{proof}
  We assume that $\forall x \in V, \mu(\omega(x) = 1) > 0$. Otherwise, according to the FKG inequality, the result is true.\smallskip

  For $k \in \llbracket 1, n \rrbracket$ we consider the decision tree $T_k$ associated to the graph walk from $\partial L_k$ and restricted to $L_n$, that computes $f = \ind_{0 \lra \partial L_n}$. The vertex $x \in L_n$ is revealed in $\omega \in \Omega$ iff it is connected to $\partial L_k$ in $\omega^x$, therefore, using the FKG inequality
  \begin{equation}
    \mu(\omega(x) = 1) \delta_x(T_k, f) \le \mu(x \lra \partial L_k) \le \mu(x \lra \partial L_{|k - \| x \|_\infty|}(x)).
  \end{equation}
  When $k$ goes from $1$ to $n$, the integer $|k - \| x \|_\infty|$ takes each value in $\llbracket 0, n-1 \rrbracket$ at most twice (except for the case $x=0$, for which we use that $0 \le \mu(0 \lra \partial L_k) \le \mu(0 \lra \partial L_{k-1})$). Hence
  \begin{equation}
    \sum\limits_{k=1}^n \delta_x(T_k, f) \le \frac 2 {\mu(\omega(x) = 1)} \sum\limits_{k=0}^{n-1} \mu(x \lra \partial L_k(x)) \le 2 \max\limits_{y \in L_n} \frac 1 {\mu(\omega(y) = 1)} \sum\limits_{k=0}^{n-1} \mu(y \lra \partial L_k(y)).
  \end{equation}
  Now we apply Theorem~\ref{thm:OSSS}.
  \begin{align}
    n \mu(0 \lra \partial L_n) (1 - \mu(0 \lra \partial L_n)) & = \sum\limits_{k=1}^n \Var(f) \\
                                                              & \le \sum\limits_{k=1}^n \sum\limits_{x \in V} \delta_x(T_k, f) \Cov(f, \omega(x)) \\
                                                              & \le 2 \sum\limits_{x \in V} \Cov(f, \omega(x)) \max\limits_{y \in L_n} \frac 1 {\mu(\omega(y) = 1)} \sum\limits_{k=0}^{n-1} \mu(y \lra \partial L_k(y)).
  \end{align}
  Therefore \eqref{eq:4.6-site} holds.
\end{proof}

\subsection{Using the OSSS inequality on bond and site percolation}

\begin{SCfigure}
  \centering
  \input{figs/site-bond}
  \caption{The new graph $G$. The vertices are white in $V_0$ (former vertices) and black in $V_1$ (former edges). The whole point of this transformation is that a vertex with exactly two neighbours is just an edge between them.}
  \label{fig:site-bond}
\end{SCfigure}
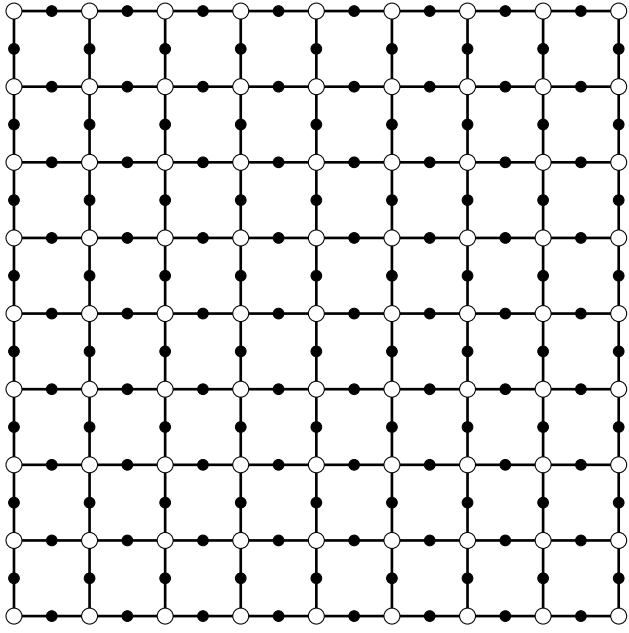

To apply Lemma~\ref{lem:4.6-site} we notice that the current model is equivalent to site percolation on a subgraph of $\Z^d$, consisting of vertices $x \in \Z^d$ with $|\{i \in \llbracket 1, d \rrbracket: x_i \equiv 1 \pmod 2\}| \le 1$ (see Figure~\ref{fig:site-bond}). We denote $G = (V, E)$ this graph. We denote $V_0 = 2 \Z^d$ and $V_1 = V \backslash V_0$. We consider site percolation on $G$, where sites are occupied independently with probability $q$ for sites in $V_0$ and $p$ for sites in $V_1$. For $x \in V$ and $n \in \N$, we denote $L_n(x) = \Lambda_n(x) \cap V$ and $\partial L_n(x) = \partial \Lambda_n(x) \cap V$. Denote $$A_n = \{0 \lra \partial L_n\}, \quad \theta_n(p, q) = \Proba(A_n) \qquad \text{ and } \qquad \theta(p, q) = \Proba(0 \lra \infty).$$
We can use Lemma~\ref{lem:4.6-site}. The boxes $L_n$ can be seen as half as large as the original boxes $\Lambda_n$, which simply changes the unexpressed constants in the final result. We apply Lemma~\ref{lem:4.6-site} in the finite graph $L_{2n}$ (so that, for $x \in L_n$ and $k \in \llbracket 0, n \rrbracket$, we have $L_k(x) \subset L_{2n}$). When computing the minimum, we have only two possible values: if $x \in V_0$, then
\begin{equation}
  \frac {\Proba(\omega(x) = 1)} {\sum\limits_{k=0}^{n-1} \Proba(x \lra \partial L_k(x))} = \frac q {\sum\limits_{k=0}^{n-1} \Proba(0 \lra \partial L_k)} .
\end{equation}
If $x \in V_1$, with neighbours $y, z$, then
\begin{equation}
  \frac {\Proba(\omega(x) = 1)} {\sum\limits_{k=0}^{n-1} \Proba(x \lra \partial L_k(x))} = \frac p {\sum\limits_{k=0}^{n-1} \Proba(x \lra \partial L_k(x))} .
\end{equation}
Note that $\Proba(x \lra \partial L_0(x)) = 1 = \Proba(0 \lra \partial L_0)$\footnote{The convention that $x \lra x$ iff $x$ is occupied may be more natural, but for the sake of simplicity we choose the convention that $x \lra x$ always holds.} and, for $k \ge 1$,
\begin{equation}
  \Proba(x \lra \partial L_k(x)) \le \Proba(y \lra \partial L_{k-1}(y)) + \Proba(z \lra \partial L_{k-1}(z)) = 2 \Proba(0 \lra \partial L_{k-1}).
\end{equation}
Therefore,
\begin{equation}
  \sum\limits_{k=0}^{n-1} \Proba(x \lra \partial L_(x)) \le \Proba(0 \lra \partial L_0) + 2 \sum\limits_{k=0}^{n-2} \Proba(0 \lra \partial L_k) \le 3 \sum\limits_{k=0}^{n-1} \Proba(0 \lra \partial L_k).
\end{equation}
We conclude that
\begin{equation}
  \min\limits_{y \in V} \frac {\Proba(\omega(y) = 1)} {\sum\limits_{k=0}^{n-1} \Proba(y \lra \partial L_k(y))} \ge \frac {p \wedge q} {3 \sum\limits_{k=0}^{n-1} \Proba(0 \lra L_k)}.
\end{equation}
Now, Lemma~\ref{lem:4.6-site} becomes
\begin{equation}
  \sum\limits_{u \in V} \Cov(\ind_{(0 \lra \partial L_n)}, \omega(u)) \ge \frac {n (p \wedge q)} {6 \sum\limits_{k=0}^{n-1} \Proba(0 \lra \partial L_k)} \Proba(0 \lra \partial L_n) (1 - \Proba(0 \lra \partial L_n)).
\end{equation}
By Theorem~\ref{thm:Russo} and Lemma~\ref{lem:cov}, 
\begin{align}
  \partial_p \theta_n & = \sum\limits_{x \in V_1} \Proba(x \pivotal A_n) = \frac 1 {p(1-p)} \sum\limits_{x \in V_1} \Cov(\ind_{A_n}, \omega(x)), \\
  \partial_q \theta_n & = \sum\limits_{x \in V_0} \Proba(x \pivotal A_n) = \frac 1 {q(1-q)} \sum\limits_{x \in V_0} \Cov(\ind_{A_n}, \omega(x)).
\end{align}
Before we can put everything together, we need inequalities relating the probability of being pivotal for vertices in~$V_0$ and $V_1$.

\subsection{Inequalities with pivotal variables}

\begin{SCfigure}
  \input{figs/edge-pivotal-so-vertex-too}
  \caption{$L_6$ (remember that boxes are twice as small) and the connected component of the origin. For the event $0 \lra \partial L_6$, the pivotal edges are $e_1$ and $e_2$. Their endpoints are pivotal as well (provided that these edges are open). Indeed, a path going through $e_1$ or $e_2$ has to go through their endpoints. However, even though both endpoints of $e_3$ are pivotal, $e_3$ itself is not.}
  \label{fig:edges-pivotal-so-vertex-too}
\end{SCfigure}
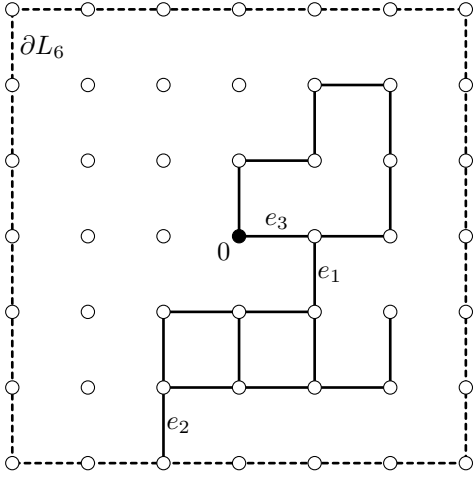

\begin{lemma}
  Let $0 < p, q < 1$ and $n \in \N^*$. Then for $e \in V_1$ with neighbours $x, y$,
  \begin{equation}
    \Proba(e \pivotal A_n) \le \frac qp (\Proba(x \pivotal A_n) + \Proba(y \pivotal A_n))
  \end{equation}
  and, for $x \in V_0$,
  \begin{equation}
    \Proba(x \pivotal A_n) \le \frac {(2d-1)p} {q (1-p)^{2d-1}} \sum\limits_{e \in \nei(x)} \Proba(e \pivotal A_n),
  \end{equation}
  where $\nei(x)$ denotes the set of vertices adjacent to vertex~$x$\footnote{Note that $\nei(x) \subset V_1$ corresponds to the set of edges adjacent to $x$ in the original graph.}.
  \label{lem:ineq-pivot-site-bond}
\end{lemma}

\begin{proof}
  We can notice that, for the events we consider, a vertex in $V_1$ can be pivotal and occupied only if at least one of its two neighbours (its "endpoints") is pivotal and occupied (see Figure~\ref{fig:edges-pivotal-so-vertex-too}). Indeed, a path in $L_n$ from $0$ to $\partial \Lambda_n$ going through $e \in V_1$ has to go through its two neighbours (or only one, and always the same, if $n$ odd and $e \in \partial L_n$).
  Therefore, if $e \in V_1$ has endpoints $x$ and $y$,
  \footnote{The inequality below is suboptimal, as one of the right-hand side terms can be removed. However, when we sum over all $e$, it make little difference. For example, with $n=1$, we must always keep $0$ so we have a $2d$ factor. This phenomenon occurs around the $d-2$ faces of the cube $L_n$ with $n$ odd.}
  \begin{equation}
    p \Proba(e \pivotal A_n) \le q \Proba(x \pivotal A_n) + q \Proba(y \pivotal A_n).
  \end{equation}
  We also have an inequality in the other direction. Let $x \in L_{n-1} \cap V_0 \backslash \{0\}$. Let $\omega \in \Omega$ where $x$ is pivotal for $A_n$. Then, in $\omega^x$, $0 \lra \partial L_n$ and all paths from $0$ to $\partial L_n$ go through $x$. Therefore, there exists a pair of occupied vertices $e, f \in \nei(x) \subset V_1$ and a vertex-self-avoiding path in $\omega^x$ from $0$ to $\partial L_n$ going through $e, x, f$. This path still exists in $\omega^x_{g \in \nei(x) \backslash \{e, f\}}$, but any path from $0$ to $\partial L_n$ has to go through $x$ and therefore through $e$ and $f$ as well, which means that they are both pivotal. Denote by $B_{e, f}$ the event that $e$ is occupied and pivotal for $A_n$ in $\omega^x_{\nei(x) \backslash \{e, f\}}$. We have
  \begin{equation}
    \Proba(x \pivotal A_n) \le \sum\limits_{e, f \in \nei(x), e \ne f} \Proba(B_{e, f})
  \end{equation}
  and, by independence and inclusion,
  \begin{multline}
    q (1-p)^{2d-2} \Proba(B_{e, f}) = \Proba(B_{e, f}, \omega(x) = 1, \forall g \in \nei(x) \backslash \{e, f\}, \omega(g) = 0) \\
    \le \Proba(e \occ, e \pivotal A_n) = p \Proba(e \pivotal A_n).
  \end{multline}
  Therefore,
  \begin{equation}
    \Proba(x \pivotal A_n) \le \frac {(2d-1)p} {q (1-p)^{2d-2}} \sum\limits_{e \in \nei(x)} \Proba(e \pivotal A_n).
    \label{eq:x-pivotal-interior}
  \end{equation}

  Let $x \in \{0\} \cup \partial L_n \cap V_0$ and $\omega \in \Omega$ such that $x$ is pivotal. Then, in $\omega^x$, there is a vertex-self-avoiding path from $0$ to $\partial L_n$ that goes through $x$. It goes through exactly one of its neighbours $e \in \nei(x) \subset V_1$. In $\omega^x_{\nei(x) \backslash \{e\}}$, we have that $e$ is occupied and pivotal. We denote $B_e$ the event that $e$ is occupied and pivotal in $\omega^x_{\nei(x) \backslash \{e\}}$. We have
  \begin{equation}
    \Proba(x \pivotal A_n) \le \sum\limits_{e \in \nei(x)} \Proba(B_e)
  \end{equation}
  and
  \begin{multline}
    q (1-p)^{2d-1} \Proba(B_e) = \Proba(B_e, \omega(x) = 1, \forall g \in \nei(x) \backslash \{e\}, \omega(g) = 0) \\
    \le \Proba(e \occ, e \pivotal A_n) = p \Proba(e \pivotal A_n).
  \end{multline}
  Therefore,
  \begin{equation}
    \Proba(x \pivotal A_n) \le \frac p {q(1-p)^{2d-1}} \sum\limits_{e \in \nei(x)} \Proba(e \pivotal A_n).
    \label{eq:x-pivotal-border}
  \end{equation}
  We combine \eqref{eq:x-pivotal-border} and \eqref{eq:x-pivotal-interior} to get the desired result.
\end{proof}

\begin{SCfigure}
  \input{figs/vertex-pivotal-so-edges-too}
  \caption{$L_6$ (remember that boxes are twice as small) and the connected component of the origin. We consider the configuration on the left. For the event $0 \lra \partial L_6$, the vertex $x$ is pivotal. If $x$ is occupied and we keep only two of the neighbouring edges, such as $e_1$ and $e_2$, we can make them pivotal (right). Note that several choices of $e_1$ and $e_2$ are possible.}
  \label{fig:vertex-pivotal-so-edges-too}
\end{SCfigure}
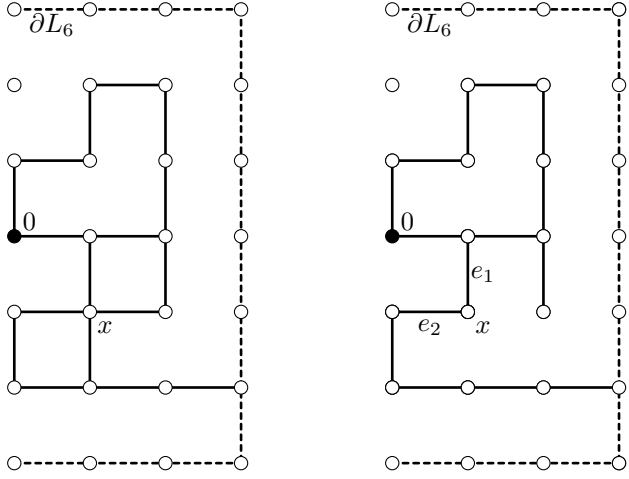

\begin{remark}
  Although not necessary and more intricate, the second inequality can be improved by not summing over all $f \in \nei(x) \backslash \{e\}$ at the cost of sometimes choosing a not so good $f$, as described below. If $p \ge \frac 1 {2d-1}$, this is an improvement. Note that if $p < \frac 1 {2d-1}$ then we are in the subcritical phase, as the BFS from $0$ in the original graph is bounded by a Galton-Watson process with offspring distribution $\mathrm{Bin}(2d-1, p)$.
\end{remark}

\begin{lemma}
  With the assumptions of Lemma~\ref{lem:ineq-pivot-site-bond},
  \begin{equation}
    \Proba(x \pivotal A_n) \le \frac p {q (1-p)^{2d-2} (p \wedge (1-p))} \sum\limits_{e \in \nei(x)} \Proba(e \pivotal A_n).
  \end{equation}
\end{lemma}

\begin{proof}
  We fix $x \in V_0 \cap \Lambda_{n-1} \backslash \{0\}$ and consider a total order on the finite set $\nei(x)$. For $e \in \nei(x) \subset V_1$ we define $$F_e = \min \{f \in \nei(x) \backslash \{e\} : e \pivotal A_n \text{ in } \omega^{\{x, e, f\}}_{\nei(x)\backslash\{e, f\}}\}.$$ If this set is empty, we define $F_e = \infty$. Then $F_e$ is independent of $\omega(x)$ and $\omega(f)$ for $f \in \nei(x)$.\smallskip

  If $\omega \in \Omega$ makes $x$ pivotal, then, as for the proof of Lemma~\ref{lem:ineq-pivot-site-bond}, it has two neighbours $e, f$ occupied and pivotal in $\omega^x_{\nei(x)\backslash\{e, f\}}$. In particular, $e$ is occupied and $F_e \ne \infty$ ($f$ is a possible value of $F_e$). Therefore,
  \begin{equation}
    \Proba(x \pivotal A_n) \le \sum\limits_{e \in \nei(x)} \Proba(F_e \ne \infty, e \occ) = p \sum\limits_{e \in \nei(x)} \Proba(F_e \ne \infty).
  \end{equation}
  Moreover, for $e \in \nei(x)$, if $x$ is occupied, $F_e \ne \infty$, $F_e$ is occupied and every vertex in $\nei(x) \backslash \{e, F_e\}$ is empty, then $e$ is pivotal (by definition of $F_e$). Thus,
  \begin{align}
    \Proba(e \pivotal A_n) & \ge \sum\limits_{f \in \nei(x) \backslash \{e\}} \Proba(F_e = f, \omega(x) = 1, \omega(f) = 1, \forall f' \in \nei(x) \backslash \{e, f\}, \omega(f') = 0) \\
                           & = \Proba(F_e \ne \infty) q p (1-p)^{2d-2}.
  \end{align}
  Hence
  \begin{equation}
    q(1-p)^{2d-2} \Proba(x \pivotal A_n) \le \sum\limits_{e \in \nei(x)} \Proba(e \pivotal A_n).
  \end{equation}
  Combining this with \eqref{eq:x-pivotal-border} ends the proof.
\end{proof}

\subsection{Phase transition along curves}

We now connect results to formulate a theorem. Let $u, v$ be continuous functions from $]0, 1[^2 \to \R_+^*$ such that $\forall p, q \in ]0, 1[, \forall n \in \N^*,$ it holds that
\begin{align}
  \partial_p \theta_n & \le u(p, q) \partial_q \theta_n, \\
  \partial_q \theta_n & \le v(p, q) \partial_p \theta_n.
\end{align}
According to Lemma~\ref{lem:ineq-pivot-site-bond} and Theorem~\ref{thm:Russo}, we can choose
\begin{align}
  u(p, q) & = 2d \frac qp, \\
  v(p, q) & = \frac {2(2d-1)p}{q(1-p)^{2d-1}}.
\end{align}

\begin{theorem}
  Let $\gamma : [0, 1] \to [0, 1]^2$ be a continuous path such that $\gamma$ is of class $C^1$ on $]0, 1[$ and $\gamma(]0, 1[) \subset ]0, 1[^2$. We assume that for all $t \in ]0, 1[$ one of the following inequalities holds:
  \begin{equation}
    \gamma_q'(t) \ge \max \{0, - u(\gamma(t)) \gamma_p'(t)\},
    \label{eq:condition-gammaq}
  \end{equation}
  or
  \begin{equation}
    \gamma_p'(t) \ge \max \{0, - v(\gamma(t)) \gamma_q'(t)\}.
    \label{eq:condition-gammap}
  \end{equation}
  Then for all $n \in \N$ the function $\theta_n \circ \gamma$ is nondecreasing. Moreover, if there exists $t_c \in ]0, 1[$ such that $\theta(\gamma(t)) = 0$ for all $0 \le t < t_c$, and  $\theta(\gamma(t)) > 0$ for all $t_c < t \le 1$, and at $t_c$ one of the aforementioned inequalities holds without equality, then
  \begin{itemize}
    \item for $0 \le t < t_c$, there exists $c_t > 0$ such that for any $n$ large enough, $\theta_n(\gamma(t)) \le \exp(-c_t n)$,
    \item there exists $c > 0$ and $\delta > 0$ such that  $\theta(\gamma(t)) \ge c (t - t_c)$ for all $t_c \le t \le t_c + \delta$.
  \end{itemize}
  \label{thm:curve}
\end{theorem}

\begin{proof}
  Let $t \in ]0, 1[$. By the chain rule, we have
  \begin{equation}
    (\theta_n \circ \gamma)'(t) = \gamma_p'(t) \partial_p \theta_n(\gamma(t)) + \gamma_q'(t) \partial_q \theta_n(\gamma(t)).
  \end{equation}
 As  $\theta_n$ is the probability of an increasing event, we have $\partial_p \theta_n, \partial_q \theta_n \ge 0$. \smallskip

  Assume that \eqref{eq:condition-gammaq} holds. Then
  \begin{align}
    (\theta_n \circ \gamma)'(t) & \ge \gamma_p'(t) \partial_p \theta_n(\gamma(t)) + \frac {\gamma_q'(t)} {u(\gamma(t))} \partial_p \theta_n(\gamma(t))) \\
                                & \ge \gamma_p'(t) \partial_p \theta_n(\gamma(t)) - \frac {u(\gamma(t)) \gamma_p'(t)} {u(\gamma(t)} \partial_p \theta_n(\gamma(t)) = 0.
  \end{align}
  The case where \eqref{eq:condition-gammap} holds is similar. Therefore, $\theta_n \circ \gamma$ is nondecreasing.\smallskip

  We shall now prove the second part of the theorem. We assume the existence of $t_c$ as described in the theorem and that \eqref{eq:condition-gammaq} holds without equality at $t_c$ (the case where it is \eqref{eq:condition-gammap} is similar).
  By the continuity of $u, \gamma, \gamma'$, there exist $\varepsilon, \delta > 0$ such that $\delta < \min \{t_c, 1-t_c\}$ and for all $t \in [t_c-\delta, t_c+\delta]$
  \begin{itemize}
    \item $\gamma_q'(t) \ge \max \{0, -u(\gamma(t)) \gamma_p'(t)\} + \varepsilon$,
    \item $\gamma(t) \in [\varepsilon, 1-\varepsilon]^2$.
  \end{itemize}
  Let $t \in [t_c-\delta, t_c+\delta]$. We use the same inequalities as above, but with some perturbations.
  \begin{align}
    & (\theta_n \circ \gamma)'(t)  = \gamma_p'(t) \partial_p \theta_n(\gamma(t)) + \gamma_q'(t) \partial_q \theta_n(\gamma(t)) \\
                                & \ge \gamma_p'(t) \partial_p \theta_n(\gamma(t)) + \frac 1 {1 + u(\gamma(t))} \varepsilon \partial_q \theta_n(\gamma(t)) + \left( \gamma_q'(t) - \frac 1 {1 + u(\gamma(t))} \varepsilon \right) \frac 1 {u(\gamma(t))} \partial_p \theta_n(\gamma(t)) \\
                                & \ge \gamma_p'(t) \partial_p \theta_n(\gamma(t)) + \frac 1 {1 + u(\gamma(t))} \varepsilon \partial_q \theta_n(\gamma(t)) + \left( -u(\gamma(t)) \gamma_p'(t) + \varepsilon - \frac 1 {1 + u(\gamma(t))} \varepsilon \right) \frac 1 {u(\gamma(t))} \partial_p \theta_n(\gamma(t)) \\
                                & = \frac \varepsilon {1 + u(\gamma(t))} (\partial_p \theta_n(\gamma(t)) + \partial_q \theta_n(\gamma(t))).
  \end{align}
  Now, we can continue the proof as in \cite{drt17}. By Theorem~\ref{thm:Russo}, Lemma~\ref{lem:cov} and Lemma~\ref{lem:4.6-site},
  \begin{align}
    (\theta_n \circ \gamma)'(t) & \ge \frac \varepsilon {1 + u(\gamma(t))} \sum\limits_{x \in V} \Proba(x \pivotal A_n) \\
                                & \ge \frac \varepsilon {1 + u(\gamma(t))} \sum\limits_{x \in V} 4 \Cov(\omega(x), \ind_{A_n}) \\
                                & \ge \frac {4 \varepsilon} {1 + u(\gamma(t))} \cdot \frac {n (\gamma_p(t) \wedge \gamma_q(t))} {6 \sum\limits_{k=0}^{n-1} \theta_k(\gamma(t))} \theta_n(\gamma(t)) (1 - \theta_n(\gamma(t))).
  \end{align}
  Note that the following bounds hold:
  \begin{itemize}
    \item $\gamma_p(t) \wedge \gamma_q(t) \ge \varepsilon > 0$
    \item $u(\gamma(t)) \le \max_{s \in [t_c - \delta, d_c + \delta]} u(\gamma(s)) < + \infty$
    \item $1 - \theta_n(\gamma(t)) \ge 1 - \theta_1(1 - \varepsilon, 1 - \varepsilon) > 0$
  \end{itemize}
  Therefore, there is a constant $c>0$ such that for all $t \in [t_c - \delta, t_c + \delta]$
  \begin{equation}
    (\theta_n \circ \gamma)'(t) \ge c \frac n {\sum\limits_{k=0}^{n-1} (\theta_k \circ \gamma)(t)} (\theta_n \circ \gamma)(t).
  \end{equation}
  Then we just need to apply Lemma~\ref{lem:3.1} to the sequence $\left(\frac 1c \theta_n \circ \gamma \right)_{n \in \N}$ (we get the exponential decay on the whole $[0, t_c[$ using the monotonicity of $\theta_n \circ \gamma$).
\end{proof}

\begin{remark}
  Denote $p_{0, c}, q_{0, c}$ the critical probabilities for the standard bond and site percolation on $\Z^d$. If $p<p_{0, c}$ or $q<q_{0, c}$ then we have exponential decay. If $p=1, q>q_{0, c}$ or $q=1, p>p_{0, c}$ then we have site or bond supercritical percolation. Therefore, we may want to extend this theorem to phase transitions occuring at $(1, q_{0, c})$ or $(p_{0, c}, 1)$ (but not necessarily along the edge of the box $[0, 1]^2$).
\smallskip

  It is sufficient that $u$ or $v$ can be extended continuously to a neighboorhood of this point (for the functions $u, v$ we found, it is the case for $u$ at both points and for $v$ at $(p_{0, c}, 1)$) and that the corresponding inequality, \eqref{eq:condition-gammaq} or \eqref{eq:condition-gammap}, holds without equality at the critical point. This second condition is almost empty for $t_c \in ]0, 1[$ because $\gamma_p(t_c) = 1 \implies \gamma_p'(t_c) = 0$ (and the same for the other component), which means that we just do standard bond or site Bernoulli percolation, with some small perturbation.\smallskip

  However on the boundary it also makes sense to have $t_c \in \{0, 1\}$. We can get $\theta(\gamma(t)) \ge ct$ for curves that start at $(1, q_{0, c})$ or $(p_{0, c}, 1)$ and exponential decay for curves that end at one of these points (provided that they respect the slope constraints). Note that exponential decay is less interesting because it can be proved using other curves, or as part of a well-defined subcritical area (see Corollary~\ref{coro:phase-sep-pq}).
\end{remark}

\subsection{Regularity of the critical curve}

Along the lines $p=p_0$, the $\theta_n$ are nondecreasing and we have a sharp phase transition.
Therefore, we define the critical curve as the function $q_c(p_0)$ equal to the $q$-component of the critical point on the line $p = p_0$, for any $p_0 \ge p_{0, c}$ (with $p_{0, c}$ the critical probability of Bernoulli bond percolation).

\begin{theorem}
  $q_c$ is decreasing and locally-Lipschitz (hence continuous). Moreover, for all $p_{0, c} < p_1 < p_2 < 1$,
  \begin{equation}
    - \int\limits_{p_1}^{p_2} u(p, q_c(p)) \diff p
    \le q_c(p_2) - q_c(p_1)
    \le - \int\limits_{p_1}^{p_2} \frac {\diff p} {v(p, q_c(p))}.
  \end{equation}
  \label{thm:qc}
\end{theorem}

\begin{coro}
  The following statements hold.
  \begin{itemize}
    \item There is exponential decay in the open set $\{(p, q) \in [0, 1]^2: (p < p_{0, c}) \vee (q < q_c(p))\}$ and $\theta > 0$ in the open set $\{(p, q) \in [0, 1]^2: (p > p_{0, c}) \wedge (q > q_c(p))\}$.
    \item Let $\gamma: [0, 1] \to [0, 1]^2$ a continuous path such that there exists $t_c \in ]0, 1[$ with $\theta(\gamma(t)) = 0$ for all $t \in [0, t_c[$, and $\theta(\gamma(t)) > 0$ for all $t \in ]t_c, 1]$. Then $\gamma_p(t_c) \ge p_{0, c}$ and $\gamma_q(t_c) = q_c(\gamma_p(t_c))$.
  \end{itemize}
  \label{coro:phase-sep-pq}
\end{coro}

\begin{remark}
  This proves the existence of well-defined and open subcritical and supercritical phases seperated by a continuous critical curve with exponential decay in the subcritical phase. This critical curve only touches the sides of the square at its endpoints.\smallskip

  Since $q_c$ is decreasing and continuous, it is a bijection $[p_{0, c}, 1] \to [q_{0, c}, 1]$ whose inverse is $q_0 \mapsto p_c(q_0)$ the $p$-component of the critical point on the line $q=q_0$.
  This function $p_c$ is also decreasing and continuous. We can also prove that it is locally Lipschitz on $]q_{0, c}, 1]$\footnote{This time, we do note include $q_{0, c}$ because we may not be able to extend $v$ to $(1, q_{0, c})$.} and that, for all $q_{0, c} < q_1 < q_2 < 1$,
  \begin{equation}
    - \int\limits_{q_1}^{q_2} v(p_c(q), q) \diff q
    \le p_c(q_2) - p_c(q_1)
    \le - \int\limits_{q_1}^{q_2} \frac {\diff q} {u(p_c(q), q)}.
  \end{equation}
\end{remark}

\begin{proof}  
  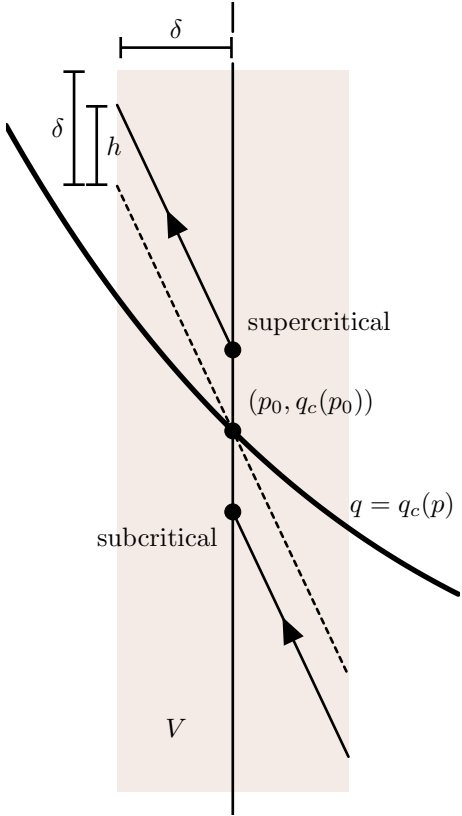
\begin{SCfigure}
    \centering
    \input{figs/lipschitz}
    \caption{The neighbourhood $V$ (the shaded area) and most objects of interest. The arrowed segments have $\theta_n$ nondecreasing along them. The one that starts above $q_c(p_0)$ has $\theta > 0$ and is therefore above the curve $q = q_c(p)$, while the one that ends below $q_c(p_0)$ has exponential decay and is therefore below the curve $q = q_c(p)$. When $h$ goes to $0$, both segments go to the dashed line.}
    \label{fig:lipschitz}
  \end{SCfigure}

    We choose a function $u$ (but not $v$) that can be continuously extended to $]0, 1]^2$ (for example, using $(p, q) \mapsto u(p, q) \wedge 2d \frac qp$). One may refer to Figure~\ref{fig:lipschitz} for a more visual description of this proof.\smallskip

  Let $p_0 \in [p_{0, c}, 1]$ and $\varepsilon > 0$. There exists $\delta > 0$ such that
  \begin{equation}
    V = [p_0 - \delta, (p_0 + \delta) \wedge 1] \times [q_c(p_0) - (u(p_0, q_c(p_0)) + \varepsilon + 1) \delta, (q_c(p_0) + (u(p_0, q_c(p_0)) + \varepsilon + 1) \delta) \wedge 1] \subset ]0, 1] \times ]0, 1]
  \end{equation}
  and $\forall (p, q) \in V, u(p, q) \le u(p_0, q_c(p_0)) + \varepsilon$.
  Consider $(p_0 - \delta) \vee p_{0, c} \le p_1 < p_0$. Let $\delta>h>0$. Consider the path
  \begin{equation}
    \gamma: t \in [-\delta, \delta] \mapsto (p_0 - t, q_c(p_0) + h + (u(p_0, q_c(p_0)) + \varepsilon) t)
  \end{equation}
  Its restriction to $\gamma^{-1}([0, 1]^2)$ is a segment in $V$, on which $u(p, q) \le u(p_0, q_c(p_0)) + \varepsilon$. We have $\gamma_p' = -1$ and $\gamma_q' = u(p_0, q_c(p_0)) + \varepsilon \ge u \circ \gamma$.
  Therefore, by Theorem~\ref{thm:curve}, $\theta \circ \gamma$ is nondecreasing on $\gamma^{-1}([0, 1]^2)$. In particular, if $\gamma(p_0 - p_1) \in [0, 1]^2$, we have
  \begin{equation}
    \theta(p_1, q_c(p_0) = \theta(\gamma(p_0 - p_1)) \ge \theta(\gamma(0)) = \theta(p_0, q_c(p_0) + h) > 0.
  \end{equation}
  Otherwise, $\gamma_q(p_0 - p_1) \ge 1 \ge q_c(p_1)$. Hence,
  \begin{align}
    q_c(p_0) \le q_c(p_1) & \le \gamma_q(p_0 - p_1) \\
                          & = q_c(p_0) + h + (u(p_0, q_c(p_0)) + \varepsilon) (p_0-p_1) \\
                          & \underset{h \to 0^+}{\longrightarrow} q_c(p_0) + (u(p_0, q_c(p_0)) + \varepsilon) (p_0-p_1).
  \end{align}
  Similarly, for $p_0 < p_1 \le (p_0 + \delta) \wedge 1$, we consider the path
  \begin{equation}
    \gamma: t \in [-\delta, \delta] \mapsto (p_0 - t, q_c(p_0) - h + (u(p_0, q_c(p_0)) + \varepsilon) t).
  \end{equation}
  Its restriction to $\gamma^{-1}([0, 1]^2)$ is a segment in $V$ with the same slope as the previous one, so, according to Theorem~\ref{thm:curve}, $\theta_n \circ \gamma$ is nondecreasing on $\gamma^{-1}([0, 1]^2)$. In particular, if $\gamma(p_0 - p_1) \in [0, 1]^2$, we have exponential decay at $\gamma(0)$ and therefore at $\gamma(p_0 - p_1)$. Otherwise, $\gamma_q(p_0 - p_1) \le 0 \le q_c(p_1)$. Hence,
  \begin{align}
    q_c(p_0) \ge q_c(p_1) & \ge \gamma_q(p_0 - p_1) \\
                          & = q_c(p_0) - h + (u(p_0, q_c(p_0)) + \varepsilon) (p_0-p_1) \\
                          & \underset{h \to 0^+}{\longrightarrow} q_c(p_0) + (u(p_0, q_c(p_0)) + \varepsilon) (p_0-p_1).
  \end{align}
  Therefore, $q_c$ is locally Lipschitz hence continuous, and
  \begin{equation}
    \liminf\limits_{h \to 0} \frac {q_c(p_0 + h) - q_c(p_0)} h \ge - u(p_0, q_c(p_0)).
  \end{equation}
  Using Fatou's lemma, for $p_{0, c} < p_1 < p_2 < 1$, we have
  \begin{equation}
    \liminf\limits_{h \to 0^+} \int\limits_{p_1}^{p_2} \frac 1h (q_c(p+h) - q_c(p)) \diff p
    \ge \int\limits_{p_1}^{p_2} \liminf\limits_{h \to 0^+} \frac 1h (q_c(p+h) - q_c(p)) \diff p
    \ge - \int\limits_{p_1}^{p_2} u(p, q_c(p)) \diff p.
  \end{equation}
  Moreover, since $q_c$ is continuous,
  \begin{equation}
    \liminf\limits_{h \to 0^+} \int\limits_{p_1}^{p_2} \frac 1h (q_c(p+h) - q_c(p)) \diff p
    = \liminf\limits_{h \to 0^+} \frac 1h \left( \int\limits_{p_2}^{p_2+h} q_c - \int\limits_{p_1}^{p_1+h} q_c \right)
    = q_c(p_2) - q_c(p_1).
  \end{equation}
  Similarly, we have the following lower bound. For $1 > p > p_{0, c}$
  \begin{equation}
    \limsup_{h \to 0^+} \frac 1h (q_c(p+h) - q_c(p)) \le - \frac 1 {v(p, q_c(p))}
  \end{equation}
  and, after we integrate between $p_1$ and $p_2$,
  \begin{equation}
    - \int\limits_{p_1}^{p_2} \frac {\diff p} {v(p, q_c(p))} \ge q_c(p_2) - q_c(p_1).
  \end{equation}
  This proves that $q_c$ is decreasing.
\end{proof}

%% file: figs/site-bond.tex
\begin{tikzpicture}[line cap=round,line join=round,>=triangle 45,x=0.5cm,y=0.5cm]
  \clip(-9,-9) rectangle (9,9);
  \draw [line width=1pt] (-8,0)-- (8,0);
  \draw [line width=1pt] (-8,2)-- (8,2);
  \draw [line width=1pt] (-8,4)-- (8,4);
  \draw [line width=1pt] (-8,6)-- (8,6);
  \draw [line width=1pt] (-8,8)-- (8,8);
  \draw [line width=1pt] (-8,-8)-- (8,-8);
  \draw [line width=1pt] (-8,-6)-- (8,-6);
  \draw [line width=1pt] (-8,-4)-- (8,-4);
  \draw [line width=1pt] (-8,-2)-- (8,-2);
  \draw [line width=1pt] (-8,8)-- (-8,-8);
  \draw [line width=1pt] (-6,-8)-- (-6,8);
  \draw [line width=1pt] (-4,8)-- (-4,-8);
  \draw [line width=1pt] (-2,8)-- (-2,-8);
  \draw [line width=1pt] (0,8)-- (0,-8);
  \draw [line width=1pt] (2,8)-- (2,-8);
  \draw [line width=1pt] (4,8)-- (4,-8);
  \draw [line width=1pt] (6,8)-- (6,-8);
  \draw [line width=1pt] (8,8)-- (8,-8);
  \begin{scriptsize}
    \draw [color=black,fill=white] (0,0) circle (3pt);
    \draw [color=black,fill=white] (2,0) circle (3pt);
    \draw [color=black,fill=white] (4,0) circle (3pt);
    \draw [color=black,fill=white] (6,0) circle (3pt);
    \draw [color=black,fill=white] (8,0) circle (3pt);
    \draw [color=black,fill=white] (-2,0) circle (3pt);
    \draw [color=black,fill=white] (-4,0) circle (3pt);
    \draw [color=black,fill=white] (-6,0) circle (3pt);
    \draw [color=black,fill=white] (-8,0) circle (3pt);
    \draw [fill=black] (-7,0) circle (2pt);
    \draw [fill=black] (-5,0) circle (2pt);
    \draw [fill=black] (-3,0) circle (2pt);
    \draw [fill=black] (-1,0) circle (2pt);
    \draw [fill=black] (1,0) circle (2pt);
    \draw [fill=black] (3,0) circle (2pt);
    \draw [fill=black] (5,0) circle (2pt);
    \draw [fill=black] (-8,1) circle (2pt);
    \draw [fill=black] (-6,1) circle (2pt);
    \draw [fill=black] (-4,1) circle (2pt);
    \draw [fill=black] (-2,1) circle (2pt);
    \draw [fill=black] (0,1) circle (2pt);
    \draw [fill=black] (2,1) circle (2pt);
    \draw [fill=black] (4,1) circle (2pt);
    \draw [fill=black] (6,1) circle (2pt);
    \draw [fill=black] (7,0) circle (2pt);
    \draw [fill=black] (8,1) circle (2pt);
    \draw [color=black,fill=white] (-8,2) circle (3pt);
    \draw [color=black,fill=white] (8,2) circle (3pt);
    \draw [color=black,fill=white] (0,2) circle (3pt);
    \draw [color=black,fill=white] (2,2) circle (3pt);
    \draw [color=black,fill=white] (4,2) circle (3pt);
    \draw [color=black,fill=white] (6,2) circle (3pt);
    \draw [color=black,fill=white] (-2,2) circle (3pt);
    \draw [color=black,fill=white] (-4,2) circle (3pt);
    \draw [color=black,fill=white] (-6,2) circle (3pt);
    \draw [fill=black] (-7,2) circle (2pt);
    \draw [fill=black] (-5,2) circle (2pt);
    \draw [fill=black] (-3,2) circle (2pt);
    \draw [fill=black] (-1,2) circle (2pt);
    \draw [fill=black] (1,2) circle (2pt);
    \draw [fill=black] (3,2) circle (2pt);
    \draw [fill=black] (5,2) circle (2pt);
    \draw [fill=black] (-8,3) circle (2pt);
    \draw [fill=black] (-6,3) circle (2pt);
    \draw [fill=black] (-4,3) circle (2pt);
    \draw [fill=black] (-2,3) circle (2pt);
    \draw [fill=black] (0,3) circle (2pt);
    \draw [fill=black] (2,3) circle (2pt);
    \draw [fill=black] (4,3) circle (2pt);
    \draw [fill=black] (6,3) circle (2pt);
    \draw [fill=black] (7,2) circle (2pt);
    \draw [fill=black] (8,3) circle (2pt);
    \draw [color=black,fill=white] (-8,4) circle (3pt);
    \draw [color=black,fill=white] (8,4) circle (3pt);
    \draw [color=black,fill=white] (-8,6) circle (3pt);
    \draw [color=black,fill=white] (8,6) circle (3pt);
    \draw [color=black,fill=white] (0,4) circle (3pt);
    \draw [color=black,fill=white] (2,4) circle (3pt);
    \draw [color=black,fill=white] (4,4) circle (3pt);
    \draw [color=black,fill=white] (6,4) circle (3pt);
    \draw [color=black,fill=white] (-2,4) circle (3pt);
    \draw [color=black,fill=white] (-4,4) circle (3pt);
    \draw [color=black,fill=white] (-6,4) circle (3pt);
    \draw [fill=black] (-7,4) circle (2pt);
    \draw [fill=black] (-5,4) circle (2pt);
    \draw [fill=black] (-3,4) circle (2pt);
    \draw [fill=black] (-1,4) circle (2pt);
    \draw [fill=black] (1,4) circle (2pt);
    \draw [fill=black] (3,4) circle (2pt);
    \draw [fill=black] (5,4) circle (2pt);
    \draw [fill=black] (-8,5) circle (2pt);
    \draw [fill=black] (-6,5) circle (2pt);
    \draw [fill=black] (-4,5) circle (2pt);
    \draw [fill=black] (-2,5) circle (2pt);
    \draw [fill=black] (0,5) circle (2pt);
    \draw [fill=black] (2,5) circle (2pt);
    \draw [fill=black] (4,5) circle (2pt);
    \draw [fill=black] (6,5) circle (2pt);
    \draw [fill=black] (7,4) circle (2pt);
    \draw [fill=black] (8,5) circle (2pt);
    \draw [color=black,fill=white] (0,6) circle (3pt);
    \draw [color=black,fill=white] (2,6) circle (3pt);
    \draw [color=black,fill=white] (4,6) circle (3pt);
    \draw [color=black,fill=white] (6,6) circle (3pt);
    \draw [color=black,fill=white] (-2,6) circle (3pt);
    \draw [color=black,fill=white] (-4,6) circle (3pt);
    \draw [color=black,fill=white] (-6,6) circle (3pt);
    \draw [fill=black] (-7,6) circle (2pt);
    \draw [fill=black] (-5,6) circle (2pt);
    \draw [fill=black] (-3,6) circle (2pt);
    \draw [fill=black] (-1,6) circle (2pt);
    \draw [fill=black] (1,6) circle (2pt);
    \draw [fill=black] (3,6) circle (2pt);
    \draw [fill=black] (5,6) circle (2pt);
    \draw [fill=black] (-8,7) circle (2pt);
    \draw [fill=black] (-6,7) circle (2pt);
    \draw [fill=black] (-4,7) circle (2pt);
    \draw [fill=black] (-2,7) circle (2pt);
    \draw [fill=black] (0,7) circle (2pt);
    \draw [fill=black] (2,7) circle (2pt);
    \draw [fill=black] (4,7) circle (2pt);
    \draw [fill=black] (6,7) circle (2pt);
    \draw [fill=black] (7,6) circle (2pt);
    \draw [fill=black] (8,7) circle (2pt);
    \draw [color=black,fill=white] (-8,8) circle (3pt);
    \draw [color=black,fill=white] (8,8) circle (3pt);
    \draw [color=black,fill=white] (0,8) circle (3pt);
    \draw [color=black,fill=white] (2,8) circle (3pt);
    \draw [color=black,fill=white] (4,8) circle (3pt);
    \draw [color=black,fill=white] (6,8) circle (3pt);
    \draw [color=black,fill=white] (-2,8) circle (3pt);
    \draw [color=black,fill=white] (-4,8) circle (3pt);
    \draw [color=black,fill=white] (-6,8) circle (3pt);
    \draw [fill=black] (-7,8) circle (2pt);
    \draw [fill=black] (-5,8) circle (2pt);
    \draw [fill=black] (-3,8) circle (2pt);
    \draw [fill=black] (-1,8) circle (2pt);
    \draw [fill=black] (1,8) circle (2pt);
    \draw [fill=black] (3,8) circle (2pt);
    \draw [fill=black] (5,8) circle (2pt);
    \draw [fill=black] (7,8) circle (2pt);
    \draw [color=black,fill=white] (-8,-8) circle (3pt);
    \draw [color=black,fill=white] (8,-8) circle (3pt);
    \draw [color=black,fill=white] (-8,-6) circle (3pt);
    \draw [color=black,fill=white] (8,-6) circle (3pt);
    \draw [color=black,fill=white] (-8,-4) circle (3pt);
    \draw [color=black,fill=white] (8,-4) circle (3pt);
    \draw [color=black,fill=white] (-8,-2) circle (3pt);
    \draw [color=black,fill=white] (8,-2) circle (3pt);
    \draw [color=black,fill=white] (0,-8) circle (3pt);
    \draw [color=black,fill=white] (2,-8) circle (3pt);
    \draw [color=black,fill=white] (4,-8) circle (3pt);
    \draw [color=black,fill=white] (6,-8) circle (3pt);
    \draw [color=black,fill=white] (-2,-8) circle (3pt);
    \draw [color=black,fill=white] (-4,-8) circle (3pt);
    \draw [color=black,fill=white] (-6,-8) circle (3pt);
    \draw [fill=black] (-7,-8) circle (2pt);
    \draw [fill=black] (-5,-8) circle (2pt);
    \draw [fill=black] (-3,-8) circle (2pt);
    \draw [fill=black] (-1,-8) circle (2pt);
    \draw [fill=black] (1,-8) circle (2pt);
    \draw [fill=black] (3,-8) circle (2pt);
    \draw [fill=black] (5,-8) circle (2pt);
    \draw [fill=black] (-8,-7) circle (2pt);
    \draw [fill=black] (-6,-7) circle (2pt);
    \draw [fill=black] (-4,-7) circle (2pt);
    \draw [fill=black] (-2,-7) circle (2pt);
    \draw [fill=black] (0,-7) circle (2pt);
    \draw [fill=black] (2,-7) circle (2pt);
    \draw [fill=black] (4,-7) circle (2pt);
    \draw [fill=black] (6,-7) circle (2pt);
    \draw [fill=black] (7,-8) circle (2pt);
    \draw [fill=black] (8,-7) circle (2pt);
    \draw [color=black,fill=white] (0,-6) circle (3pt);
    \draw [color=black,fill=white] (2,-6) circle (3pt);
    \draw [color=black,fill=white] (4,-6) circle (3pt);
    \draw [color=black,fill=white] (6,-6) circle (3pt);
    \draw [color=black,fill=white] (-2,-6) circle (3pt);
    \draw [color=black,fill=white] (-4,-6) circle (3pt);
    \draw [color=black,fill=white] (-6,-6) circle (3pt);
    \draw [fill=black] (-7,-6) circle (2pt);
    \draw [fill=black] (-5,-6) circle (2pt);
    \draw [fill=black] (-3,-6) circle (2pt);
    \draw [fill=black] (-1,-6) circle (2pt);
    \draw [fill=black] (1,-6) circle (2pt);
    \draw [fill=black] (3,-6) circle (2pt);
    \draw [fill=black] (5,-6) circle (2pt);
    \draw [fill=black] (-8,-5) circle (2pt);
    \draw [fill=black] (-6,-5) circle (2pt);
    \draw [fill=black] (-4,-5) circle (2pt);
    \draw [fill=black] (-2,-5) circle (2pt);
    \draw [fill=black] (0,-5) circle (2pt);
    \draw [fill=black] (2,-5) circle (2pt);
    \draw [fill=black] (4,-5) circle (2pt);
    \draw [fill=black] (6,-5) circle (2pt);
    \draw [fill=black] (7,-6) circle (2pt);
    \draw [fill=black] (8,-5) circle (2pt);
    \draw [color=black,fill=white] (0,-4) circle (3pt);
    \draw [color=black,fill=white] (2,-4) circle (3pt);
    \draw [color=black,fill=white] (4,-4) circle (3pt);
    \draw [color=black,fill=white] (6,-4) circle (3pt);
    \draw [color=black,fill=white] (-2,-4) circle (3pt);
    \draw [color=black,fill=white] (-4,-4) circle (3pt);
    \draw [color=black,fill=white] (-6,-4) circle (3pt);
    \draw [fill=black] (-7,-4) circle (2pt);
    \draw [fill=black] (-5,-4) circle (2pt);
    \draw [fill=black] (-3,-4) circle (2pt);
    \draw [fill=black] (-1,-4) circle (2pt);
    \draw [fill=black] (1,-4) circle (2pt);
    \draw [fill=black] (3,-4) circle (2pt);
    \draw [fill=black] (5,-4) circle (2pt);
    \draw [fill=black] (-8,-3) circle (2pt);
    \draw [fill=black] (-6,-3) circle (2pt);
    \draw [fill=black] (-4,-3) circle (2pt);
    \draw [fill=black] (-2,-3) circle (2pt);
    \draw [fill=black] (0,-3) circle (2pt);
    \draw [fill=black] (2,-3) circle (2pt);
    \draw [fill=black] (4,-3) circle (2pt);
    \draw [fill=black] (6,-3) circle (2pt);
    \draw [fill=black] (7,-4) circle (2pt);
    \draw [fill=black] (8,-3) circle (2pt);
    \draw [color=black,fill=white] (0,-2) circle (3pt);
    \draw [color=black,fill=white] (2,-2) circle (3pt);
    \draw [color=black,fill=white] (4,-2) circle (3pt);
    \draw [color=black,fill=white] (6,-2) circle (3pt);
    \draw [color=black,fill=white] (-2,-2) circle (3pt);
    \draw [color=black,fill=white] (-4,-2) circle (3pt);
    \draw [color=black,fill=white] (-6,-2) circle (3pt);
    \draw [fill=black] (-7,-2) circle (2pt);
    \draw [fill=black] (-5,-2) circle (2pt);
    \draw [fill=black] (-3,-2) circle (2pt);
    \draw [fill=black] (-1,-2) circle (2pt);
    \draw [fill=black] (1,-2) circle (2pt);
    \draw [fill=black] (3,-2) circle (2pt);
    \draw [fill=black] (5,-2) circle (2pt);
    \draw [fill=black] (-8,-1) circle (2pt);
    \draw [fill=black] (-6,-1) circle (2pt);
    \draw [fill=black] (-4,-1) circle (2pt);
    \draw [fill=black] (-2,-1) circle (2pt);
    \draw [fill=black] (0,-1) circle (2pt);
    \draw [fill=black] (2,-1) circle (2pt);
    \draw [fill=black] (4,-1) circle (2pt);
    \draw [fill=black] (6,-1) circle (2pt);
    \draw [fill=black] (7,-2) circle (2pt);
    \draw [fill=black] (8,-1) circle (2pt);
  \end{scriptsize}
\end{tikzpicture}

%% file: figs/edge-pivotal-so-vertex-too.tex
\centering
\begin{tikzpicture}[line cap=round,line join=round,>=triangle 45,x=1cm,y=1cm,font=\normalsize]
  \clip(-4.2,-0.2) rectangle (2.2,6.2);
  \draw [line width=1pt,dash pattern=on 2pt off 2pt] (-4,6)-- (2,6);
  \draw [line width=1pt,dash pattern=on 2pt off 2pt] (2,6)-- (2,0);
  \draw [line width=1pt,dash pattern=on 2pt off 2pt] (2,0)-- (-4,0);
  \draw [line width=1pt,dash pattern=on 2pt off 2pt] (-4,0)-- (-4,6);
  \draw [line width=1pt] (-1,3)-- (-1,4);
  \draw [line width=1pt] (-1,4)-- (0,4);
  \draw [line width=1pt] (0,4)-- (0,5);
  \draw [line width=1pt] (0,5)-- (1,5);
  \draw [line width=1pt] (1,5)-- (1,4);
  \draw [line width=1pt] (0,3)-- (0,2);
  \draw [line width=1pt] (0,2)-- (-1,2);
  \draw [line width=1pt] (-1,2)-- (-1,1);
  \draw [line width=1pt] (-1,1)-- (-2,1);
  \draw [line width=1pt] (-1,2)-- (-2,2);
  \draw [line width=1pt] (-2,2)-- (-2,1);
  \draw [line width=1pt] (-2,1)-- (-2,0);
  \draw [line width=1pt] (0,3)-- (1,3);
  \draw [line width=1pt] (1,4)-- (1,3);
  \draw [line width=1pt] (-1,3)-- (0,3);
  \draw [line width=1pt] (1,2)-- (1,1);
  \draw [line width=1pt] (1,1)-- (0,1);
  \draw [line width=1pt] (0,1)-- (0,2);
  \draw [line width=1pt] (0,1)-- (-1,1);
  \begin{scriptsize}
    \draw [fill=ffffff] (-4,0) circle (2.5pt);
    \draw [fill=ffffff] (-3,0) circle (2.5pt);
    \draw [fill=ffffff] (-2,0) circle (2.5pt);
    \draw [fill=ffffff] (-1,0) circle (2.5pt);
    \draw [fill=ffffff] (0,0) circle (2.5pt);
    \draw [fill=ffffff] (1,0) circle (2.5pt);
    \draw [fill=ffffff] (-4,1) circle (2.5pt);
    \draw [fill=ffffff] (-3,1) circle (2.5pt);
    \draw [fill=ffffff] (-2,1) circle (2.5pt);
    \draw [fill=ffffff] (-1,1) circle (2.5pt);
    \draw [fill=ffffff] (0,1) circle (2.5pt);
    \draw [fill=ffffff] (1,1) circle (2.5pt);
    \draw [fill=ffffff] (1,2) circle (2.5pt);
    \draw [fill=ffffff] (0,2) circle (2.5pt);
    \draw [fill=ffffff] (-1,2) circle (2.5pt);
    \draw [fill=ffffff] (-2,2) circle (2.5pt);
    \draw [fill=ffffff] (-3,2) circle (2.5pt);
    \draw [fill=ffffff] (-4,2) circle (2.5pt);
    \draw [fill=ffffff] (-4,3) circle (2.5pt);
    \draw [fill=ffffff] (-4,4) circle (2.5pt);
    \draw [fill=ffffff] (-4,5) circle (2.5pt);
    \draw [fill=ffffff] (-3,5) circle (2.5pt);
    \draw [fill=ffffff] (-3,4) circle (2.5pt);
    \draw [fill=ffffff] (-3,3) circle (2.5pt);
    \draw [fill=ffffff] (-2,5) circle (2.5pt);
    \draw [fill=ffffff] (-2,4) circle (2.5pt);
    \draw [fill=ffffff] (-2,3) circle (2.5pt);
    \draw [fill=ffffff] (-1,5) circle (2.5pt);
    \draw [fill=ffffff] (-1,4) circle (2.5pt);
    \draw [fill=black] (-1,3) circle (2.5pt);
    \draw[color=black] (-1.2,2.8) node {$0$};
    \draw [fill=ffffff] (0,5) circle (2.5pt);
    \draw [fill=ffffff] (0,4) circle (2.5pt);
    \draw [fill=ffffff] (0,3) circle (2.5pt);
    \draw [fill=ffffff] (1,5) circle (2.5pt);
    \draw [fill=ffffff] (1,4) circle (2.5pt);
    \draw [fill=ffffff] (1,3) circle (2.5pt);
    \draw [fill=ffffff] (-4,6) circle (2.5pt);
    \draw [fill=ffffff] (-3,6) circle (2.5pt);
    \draw [fill=ffffff] (-2,6) circle (2.5pt);
    \draw [fill=ffffff] (-1,6) circle (2.5pt);
    \draw [fill=ffffff] (0,6) circle (2.5pt);
    \draw [fill=ffffff] (1,6) circle (2.5pt);
    \draw [fill=ffffff] (2,6) circle (2.5pt);
    \draw [fill=ffffff] (2,5) circle (2.5pt);
    \draw [fill=ffffff] (2,4) circle (2.5pt);
    \draw [fill=ffffff] (2,3) circle (2.5pt);
    \draw [fill=ffffff] (2,2) circle (2.5pt);
    \draw [fill=ffffff] (2,1) circle (2.5pt);
    \draw [fill=ffffff] (2,0) circle (2.5pt);
    \draw[color=black] (0.2,2.5) node {$e_1$};
    \draw[color=black] (-1.8,0.5) node {$e_2$};
    \draw[color=black] (-0.5,3.2) node {$e_3$};
    \draw[color=black] (-3.6, 5.5) node {$\partial L_6$};
  \end{scriptsize}
\end{tikzpicture}

%% file: figs/vertex-pivotal-so-edges-too.tex
\centering
\begin{tikzpicture}[line cap=round,line join=round,>=triangle 45,x=1cm,y=1cm,font=\normalsize]
  \clip(-1.2,-0.2) rectangle (7.2,6.2);
  \draw [line width=1pt,dash pattern=on 2pt off 2pt] (2,6)-- (2,0);
  \draw [line width=1pt] (-1,3)-- (-1,4);
  \draw [line width=1pt] (-1,4)-- (0,4);
  \draw [line width=1pt] (0,4)-- (0,5);
  \draw [line width=1pt] (0,5)-- (1,5);
  \draw [line width=1pt] (1,5)-- (1,4);
  \draw [line width=1pt] (0,3)-- (0,2);
  \draw [line width=1pt] (-1,2)-- (-1,1);
  \draw [line width=1pt] (0,3)-- (1,3);
  \draw [line width=1pt] (1,4)-- (1,3);
  \draw [line width=1pt] (-1,3)-- (0,3);
  \draw [line width=1pt] (1,1)-- (0,1);
  \draw [line width=1pt] (0,1)-- (0,2);
  \draw [line width=1pt] (0,1)-- (-1,1);
  \draw [line width=1pt,dash pattern=on 2pt off 2pt] (7,6)-- (7,0);
  \draw [line width=1pt] (4,3)-- (4,4);
  \draw [line width=1pt] (4,4)-- (5,4);
  \draw [line width=1pt] (5,4)-- (5,5);
  \draw [line width=1pt] (5,5)-- (6,5);
  \draw [line width=1pt] (6,5)-- (6,4);
  \draw [line width=1pt] (5,2)-- (4,2);
  \draw [line width=1pt] (4,2)-- (4,1);
  \draw [line width=1pt] (5,3)-- (6,3);
  \draw [line width=1pt] (6,4)-- (6,3);
  \draw [line width=1pt] (4,3)-- (5,3);
  \draw [line width=1pt] (6,1)-- (5,1);
  \draw [line width=1pt] (5,1)-- (4,1);
  \draw [line width=1pt] (-1,2)-- (1,2);
  \draw [line width=1pt] (5,3)-- (5,2);
  \draw [line width=1pt] (1,2)-- (1,3);
  \draw [line width=1pt] (6,3)-- (6,2);
  \draw [line width=1pt] (1,1)-- (2,1);
  \draw [line width=1pt] (6,1)-- (7,1);
  \draw [line width=1pt,dash pattern=on 2pt off 2pt] (-1,6)-- (2,6);
  \draw [line width=1pt,dash pattern=on 2pt off 2pt] (-1,0)-- (2,0);
  \draw [line width=1pt,dash pattern=on 2pt off 2pt] (4,6)-- (7,6);
  \draw [line width=1pt,dash pattern=on 2pt off 2pt] (4,0)-- (7,0);
  \begin{scriptsize}
    \draw [fill=ffffff] (-1,0) circle (2.5pt);
    \draw [fill=ffffff] (0,0) circle (2.5pt);
    \draw [fill=ffffff] (1,0) circle (2.5pt);
    \draw [fill=ffffff] (-1,1) circle (2.5pt);
    \draw [fill=ffffff] (0,1) circle (2.5pt);
    \draw [fill=ffffff] (1,1) circle (2.5pt);
    \draw [fill=ffffff] (1,2) circle (2.5pt);
    \draw [fill=ffffff] (0,2) circle (2.5pt);
    \draw [fill=ffffff] (-1,2) circle (2.5pt);
    \draw [fill=ffffff] (-1,5) circle (2.5pt);
    \draw [fill=ffffff] (-1,4) circle (2.5pt);
    \draw [fill=black] (-1,3) circle (2.5pt);
    \draw[color=black] (-0.8,3.2) node {$0$};
    \draw[color=black] (0.2,1.8) node {$x$};
    \draw[color=black] (-0.5, 5.8) node {$\partial L_6$};
    \draw [fill=ffffff] (0,5) circle (2.5pt);
    \draw [fill=ffffff] (0,4) circle (2.5pt);
    \draw [fill=ffffff] (0,3) circle (2.5pt);
    \draw [fill=ffffff] (1,5) circle (2.5pt);
    \draw [fill=ffffff] (1,4) circle (2.5pt);
    \draw [fill=ffffff] (1,3) circle (2.5pt);
    \draw [fill=ffffff] (-1,6) circle (2.5pt);
    \draw [fill=ffffff] (0,6) circle (2.5pt);
    \draw [fill=ffffff] (1,6) circle (2.5pt);
    \draw [fill=ffffff] (2,6) circle (2.5pt);
    \draw [fill=ffffff] (2,5) circle (2.5pt);
    \draw [fill=ffffff] (2,4) circle (2.5pt);
    \draw [fill=ffffff] (2,3) circle (2.5pt);
    \draw [fill=ffffff] (2,2) circle (2.5pt);
    \draw [fill=ffffff] (2,1) circle (2.5pt);
    \draw [fill=ffffff] (2,0) circle (2.5pt);
    \draw [fill=ffffff] (7,6) circle (2.5pt);
    \draw [fill=ffffff] (7,6) circle (2.5pt);
    \draw [fill=ffffff] (7,0) circle (2.5pt);
    \draw [fill=ffffff] (7,0) circle (2.5pt);
    \draw [fill=ffffff] (4,4) circle (2.5pt);
    \draw [fill=ffffff] (4,4) circle (2.5pt);
    \draw [fill=ffffff] (5,4) circle (2.5pt);
    \draw [fill=ffffff] (5,4) circle (2.5pt);
    \draw [fill=ffffff] (5,5) circle (2.5pt);
    \draw [fill=ffffff] (5,5) circle (2.5pt);
    \draw [fill=ffffff] (6,5) circle (2.5pt);
    \draw [fill=ffffff] (6,5) circle (2.5pt);
    \draw [fill=ffffff] (6,4) circle (2.5pt);
    \draw [fill=ffffff] (5,3) circle (2.5pt);
    \draw [fill=ffffff] (5,2) circle (2.5pt);
    \draw [fill=ffffff] (5,2) circle (2.5pt);
    \draw [fill=ffffff] (4,2) circle (2.5pt);
    \draw [fill=ffffff] (4,2) circle (2.5pt);
    \draw [fill=ffffff] (4,1) circle (2.5pt);
    \draw [fill=ffffff] (4,1) circle (2.5pt);
    \draw [fill=ffffff] (4,2) circle (2.5pt);
    \draw [fill=ffffff] (5,3) circle (2.5pt);
    \draw [fill=ffffff] (6,3) circle (2.5pt);
    \draw [fill=ffffff] (6,4) circle (2.5pt);
    \draw [fill=ffffff] (6,3) circle (2.5pt);
    \draw [fill=black] (4,3) circle (2.5pt);
    \draw[color=black] (4.2,3.2) node {$0$};
    \draw[color=black] (5.2,1.8) node {$x$};
    \draw[color=black] (5.2,2.5) node {$e_1$};
    \draw[color=black] (4.5,1.8) node {$e_2$};
    \draw[color=black] (4.5, 5.8) node {$\partial L_6$};
    \draw [fill=ffffff] (5,3) circle (2.5pt);
    \draw [fill=ffffff] (5,2) circle (2.5pt);
    \draw [fill=ffffff] (6,2) circle (2.5pt);
    \draw [fill=ffffff] (6,2) circle (2.5pt);
    \draw [fill=ffffff] (6,1) circle (2.5pt);
    \draw [fill=ffffff] (6,1) circle (2.5pt);
    \draw [fill=ffffff] (5,1) circle (2.5pt);
    \draw [fill=ffffff] (5,1) circle (2.5pt);
    \draw [fill=ffffff] (5,2) circle (2.5pt);
    \draw [fill=ffffff] (5,1) circle (2.5pt);
    \draw [fill=ffffff] (4,1) circle (2.5pt);
    \draw [fill=ffffff] (4,0) circle (2.5pt);
    \draw [fill=ffffff] (5,0) circle (2.5pt);
    \draw [fill=ffffff] (6,0) circle (2.5pt);
    \draw [fill=ffffff] (4,5) circle (2.5pt);
    \draw [fill=ffffff] (4,6) circle (2.5pt);
    \draw [fill=ffffff] (5,6) circle (2.5pt);
    \draw [fill=ffffff] (6,6) circle (2.5pt);
    \draw [fill=ffffff] (7,5) circle (2.5pt);
    \draw [fill=ffffff] (7,4) circle (2.5pt);
    \draw [fill=ffffff] (7,3) circle (2.5pt);
    \draw [fill=ffffff] (7,2) circle (2.5pt);
    \draw [fill=ffffff] (7,1) circle (2.5pt);
  \end{scriptsize}
\end{tikzpicture}

%% file: figs/lipschitz.tex
\definecolor{zzttqq}{rgb}{0.6,0.2,0}
\begin{tikzpicture}[line cap=round,line join=round,>=triangle 45,x=30cm,y=30cm,font=\normalsize]
  \clip(0.10,0.462144) rectangle (0.30,0.82);
  \fill[line width=2pt,color=zzttqq,fill=zzttqq,fill opacity=0.10000000149011612] (0.149,0.79) -- (0.251,0.79) -- (0.251,0.4721440000000002) -- (0.149,0.4721440000000001) -- cycle;
  \draw[line width=2pt,smooth,samples=100,domain=0.1:0.3] plot(\x,{0.5*((\x)-1)^(6)+0.5}) node [right, label={[label distance=0.7cm]93:$q = q_c(p)$}] {};
  \draw [line width=1pt] (0.2,0)-- (0.2,0.793);
  \draw [line width=1pt] (0.2,0.807)-- (0.2,1);
  \draw [line width=1pt] (0.2,0.6667720000000001)-- (0.149,0.7747) node[pos=0.5, sloped, shape=isosceles triangle, fill=black, minimum width=5pt, minimum height=10pt, inner sep=0pt, shape border rotate=180] {};
  \draw [line width=1pt] (0.251,0.4874440000000001)-- (0.2,0.595372) node[pos=0.5, sloped, shape=isosceles triangle, fill=black, minimum width=5pt, minimum height=10pt, inner sep=0pt, shape border rotate=180] {};
  \draw [line width=1pt,dash pattern=on 2pt off 2pt] (0.149,0.739)-- (0.251,0.5231440000000002);
  \draw [line width=1pt,Bar-Bar] (0.14,0.739) -- (0.14,0.7747) node [pos=0.5, label={[label distance=-4pt]right:$h$}] {};
  \draw [line width=1pt,Bar-Bar] (0.13,0.739) -- (0.13,0.79) node [pos=0.5, label={[label distance=-4pt]left:$\delta$}] {};
  \draw [line width=1pt,Bar-Bar] (0.15,0.8) -- (0.2,0.8) node [pos=0.5, label={[label distance=-4pt]above:$\delta$}] {};
  \node at (0.175,0.5) {$V$};
  \begin{scriptsize}
    \draw [fill=black] (0.2,0.6310720000000001) circle (3pt) node [label={above right:$(p_0, q_c(p_0))$}] {};
    \draw [fill=black] (0.2,0.6667720000000001) circle (3pt) node [label={above right:supercritical}] {};
    \draw [fill=black] (0.2,0.595372) circle (3pt) node [label={below left:subcritical}] {};
  \end{scriptsize}
\end{tikzpicture}

%% file: sections/finite-weights.tex
\section{The case of finitely many weights}
\label{sec:fw}

In this section, we prove that if there is a phase transition for the model below, with parameter $\beta \in \R_+^*$, then this phase transition is sharp.

\begin{definition}
  Let $M, N \in \N^*$. Let $h : (m, \beta) \in \llbracket 0, M \rrbracket \times \R_+^* \mapsto h_m(\beta) \in [0, 1[$ and $\zeta : \llbracket 0, N \rrbracket^2 \to \llbracket 0, M \rrbracket$. We assume that $\zeta$ and $h$ have the following properties.
  \begin{itemize}
    \item $\forall (n_1, n_2) \in \llbracket 0, N \rrbracket^2, \zeta(n_1, n_2) = \zeta(n_2, n_1)$
    \item $\zeta$ is nondecreasing with respect to the product order.
    \item $0 = h_0 < h_1 < \ldots < h_M$ and for all $m \in \llbracket 1, M \rrbracket$, $h_m$ is increasing and of class $C^1$.
    \item For all $m \in \llbracket 1, M \rrbracket, h_m' > \frac {1 - h_m} {1 - h_{m-1}} h_{m-1}'$.\footnote{The meaning of this inequality will made be clear in Remark~\ref{rem:Bernoulli-compat}.}
  \end{itemize}
  Let $q_\nu$ be a probability distribution on $\llbracket 0, N \rrbracket$ such that $q_\nu(N) > 0$. Let
  \begin{equation}
    (\Omega, \mathcal F, \Proba) = (\llbracket 0, N \rrbracket, \mathcal P(\llbracket 0, N \rrbracket), q_\nu)^{\otimes \Z^d} \otimes ([0, 1], \mathcal B([0, 1]), \mu)^{\otimes E},
  \end{equation}
  where $\mu$ is the Lebesgue measure on $[0, 1]$. For a configuration $(\nu, U) \in \Omega$ the edge $e \in E$ is open if $$h_{\zeta(\nu_x, \nu_y)}(\beta) \ge U_e.$$
  \label{def:fw-natural}
\end{definition}

\begin{remark}
  This model contains at most the edges of standard Bernoulli bond percolation with edge probability $h_M(\beta)$, therefore if $h_M(\beta) \underset{\beta \to 0} \longrightarrow 0$ then there is a subcritical phase.\smallskip

  If there is a number $n_0 \in \llbracket 0, N \rrbracket$ such that $\zeta(n_0, n_0) \ge 1$ then this model contains at least the edges of site and bond percolation with parameter $q_\nu(\llbracket n_0, N \rrbracket)$ for the sites and $h_1(\beta)$ for the bonds. Furthermore, if \smash{$h_1(\beta) \underset{\beta \to +\infty} \longrightarrow 1$} and $q_\nu(\llbracket n_0, N \rrbracket) > q_{0, c}$ then there is a supercritical phase (see Section~\ref{sec:site-bond} and Theorem~\ref{thm:qc}).\smallskip

  These criteria do not characterize the existence of a phase transition and can be improved. However, they justify that there are many examples of this model with a phase transition.\smallskip
  
  A typical example would be
  $h(m,\beta)=1-e^{-\beta m}$ and 
  $\zeta(n_1,n_2)=n_1 n_2$
  with $q_\nu$ such that $q_\nu(0) < 1 - q_{0, c}$.
\end{remark}

\subsection{Redefining using Bernoulli variables}

In order to use our usual set of tools (especially Theorem~\ref{thm:OSSS}), we will describe this model using a family of independent Bernoulli random variables.

\begin{definition}
  Let $\Omega = \{0, 1\}^{E \times \llbracket 1, M \rrbracket} \times \{0, 1\}^{\Z^d \times \llbracket 1, N \rrbracket}$. We define, for $e \in E$ and $l \in \llbracket 1, M \rrbracket$,
  \begin{equation}
    Y_{e, l}: (y, z) \in \Omega \mapsto y(e, l) \in \{0, 1\}
  \end{equation}
  and, for $x \in \Z^d$ and $k \in \llbracket 1, N \rrbracket$,
  \begin{equation}
    Z_{x, k}: (y, z) \in \Omega \mapsto z(x, k) \in \{0, 1\}.
  \end{equation}
  We call these Bernoulli variables base variables. We define for $l \in \llbracket 1, M \rrbracket$ the function
  \begin{equation}
    p_l: \beta \in \R_+^* \mapsto 1 - \frac {1 - h_l(\beta)} {1 - h_{l-1}(\beta)} \in [0, 1].
  \end{equation}
  Let $\Proba$ be the product probability measure on $\Omega$ such that
  \begin{equation}
    \Proba(Y_{e, l} = 1) = p_l(\beta) \text{ for all $e \in E$,$l \in \llbracket 1, M \rrbracket$} 
  \end{equation}
  and
  \begin{equation}
    \Proba(Z_{x, k} = 1) = \frac {q_\nu(\llbracket k, N \rrbracket)} {q_\nu(\llbracket k-1, N \rrbracket)}
    \text{ for all $x \in \Z^d$, $k \in \llbracket 1, N \rrbracket$.}
  \end{equation}
  We define
  \begin{equation}
    \nu_x = \min \{n \in \llbracket 0, N \rrbracket : \forall 1 \le k \le n, Z_{x, k} = 1 \}
  \end{equation}
  and we say that an edge $e = \{x, y\}$ is open if
  \begin{equation}
    \max\limits_{1 \le l \le \zeta(\nu_x, \nu_y)} Y_{e, l} = 1.
  \end{equation}
  \label{def:Bernoulli-split}
\end{definition}

\begin{remark}
  Note that the last assumption on $h$ is equivalent to $p_l' > 0$. Moreover, relating this to the original model, this assumption is true for functions $h$ of the form
  \begin{equation}
    h_m(\beta) = 1 - \exp(- \beta c_m),
  \end{equation}
  where $(c_m)$ is an increasing sequence of real numbers with $c_0 = 0$. Indeed, in this case $p_l(\beta) = 1 - \exp(-\beta (c_l - c_{l-1}))$.
  \label{rem:Bernoulli-compat}
\end{remark}

\begin{prop}
  The distributions of weights and open edges in Definition~\ref{def:Bernoulli-split} and Definition~\ref{def:fw-natural} are identical.
\end{prop}

\begin{proof}
  We consider the model in Definition~\ref{def:Bernoulli-split} and prove that it is equivalent to the one defined in Definition~\ref{def:fw-natural}.\smallskip

  The weights $(\nu_x)$ depend on pairwise disjoints sets of base variables, and base variables are independent. Therefore the weights are independent. Moreover, for $x \in \Z^d$ and $k \in \llbracket 0, N \rrbracket$
  \begin{equation}
    \Proba(\nu_x \ge k)
    = \Proba(Z_{x, 1} = \ldots = Z_{x, k} = 1)
    = \prod\limits_{i=1}^k \frac {q_\nu(\llbracket i, N \rrbracket)} {q_\nu(\llbracket i-1, N \rrbracket)}
    = q_\nu(\llbracket k, N \rrbracket).
  \end{equation}
  Thus, the weights have the desired distribution. Now, we need to prove that the distribution of open edges given the weights is correct. Since the states of edges are independent given $\nu$ (as they should) we just need to compute the probability
  \begin{equation}
    \Proba(e \text{ closed} \mid \nu)
    = \Proba(Y_{e, 1} = \ldots = Y_{e, \zeta(\nu_x, \nu_y)} = 0)
    = \prod\limits_{l=1}^{\zeta(\nu_x, \nu_y)} \frac {1 - h_l(\beta)} {1 - h_{l-1}(\beta)}
    = 1 - h_{\zeta(\nu_x, \nu_y)}(\beta).
  \end{equation}
  Therefore, these two descriptions define the same model.
\end{proof}

\begin{remark}
  The event that an edge is open is increasing.
  More generally, any event or random variable that can be defined using only the states of edges and that is increasing in $\{0, 1\}^E$ is also increasing in the probability space from Definition~\ref{def:Bernoulli-split}.
\end{remark}

From now on, we consider the model as defined in Definition~\ref{def:Bernoulli-split} and denote $A_n = \{0 \lra \partial \Lambda_n\}$ and $\theta_n = \Proba(A_n)$.

\subsection{Applying the OSSS inequality}

Let us prove an alternative to Lemma~\ref{lem:3.2}.

\begin{lemma}
  Let $n \in \N^*$.
  \begin{equation}
    \frac n  {\sum\limits_{k=0}^{n-1} \theta_k} \theta_n (1 - \theta_n) \le 4 \sum\limits_{e \in E} \sum\limits_{l=1}^M \Cov(Y_{e, l}, \ind_{A_n}) + 4d \sum\limits_{x \in \Z^d} \sum\limits_{l=1}^N \Cov(Z_{x, l}, \ind_{A_n})
  \end{equation}
  \label{lem:4.6-fw}
\end{lemma}

\begin{proof}
  Note that many of these terms are equal to $0$, so we can work in the finite subgraph $\Lambda_{2n}$. This enables us to use Theorem~\ref{thm:OSSS}.
  For $k \in \llbracket 1, n \rrbracket$ we denote $T_k$ the decision tree associated to a graph walk starting with all vertices in $\partial \Lambda_k$. When this decision tree decides to uncover the state of an edge, it reveals the state of all base variables linked to this edge and its endpoints.\footnote{This is not optimal, but enough for the scope of this work.}\smallskip

  For a base random variable $X$ we denote $\delta_X(T_k) = \Proba(X \text{ revealed by } T_k)$.\footnote{For full consistency with the usual notations, we should denote $\delta_u(T_k) = \Proba(u \text{ revealed by } T_k)$ with $u \in (\Z^d \times \llbracket 1, N \rrbracket) \sqcup (E \times \llbracket 1, M \rrbracket)$ such that $X = \omega(u)$. However, we retain the labels $Y$ and $Z$ for the sake of clearly separating edge-related and vertex-related components.}
  As in \cite[Lemma 3.2]{drt17}, if an edge is revealed then one of its endpoints is connected to $\partial \Lambda_k$. Hence for an edge $e = \{x, y\} \subset \Lambda_n$ and $l \in \llbracket 1, M \rrbracket$,
  \begin{equation}
    \delta_{Y_{e, l}}(T_k) \le \Proba(x \lra \partial \Lambda_k) + \Proba(y \lra \partial \Lambda_k).
    \label{eq:yel-ersatz}
  \end{equation}
  Moreover, a vertex $x \in \Lambda_n$ is revealed if and only if one of its neighbours is connected to $\partial \Lambda_k$ or $x \in \partial \Lambda_k$ itself. Since $d \ge 2$ and $k \ge 1$, the second condition implies the first (any vertex in $\partial \Lambda_k$ has a neighbour in $\partial \Lambda_k$). For $l \in \llbracket 1, N \rrbracket$, we have
  \begin{equation}
    \delta_{Z_{x, l}}(T_k) \le \sum\limits_{y \in \nei(x) \cap \Lambda_n} \Proba(y \lra \partial \Lambda_k).
    \label{eq:zxl-ersatz}
  \end{equation}
  Now, given $x \in \Lambda_n$,
  \begin{equation}
    \sum\limits_{k=1}^n \Proba(x \lra \partial \Lambda_k) \le \sum\limits_{k=1}^n \Proba(x \lra \partial \Lambda_{|k - d(0, x)|}(x)).
  \end{equation}
  For $x \ne 0$, the integer $|k - d(0, x)|$ takes each value in $\llbracket 0, n-1 \rrbracket$ at most two times. For $x=0$ we notice that $\theta_k \le \theta_{k-1}$. From these two facts, we get that
  \begin{equation}
    \sum\limits_{k=1}^n \Proba(x \lra \partial \Lambda_k) \le 2 \sum\limits_{l=0}^{n-1} \Proba(x \lra \partial \Lambda_l(x)).
  \end{equation}
  Thanks to the translation invariance,
  \begin{equation}
    \sum\limits_{k=1}^n \Proba(x \lra \partial \Lambda_k) \le 2 \sum\limits_{l=0}^{n-1} \theta_l.
    \label{eq:sum-x-lra-k}
  \end{equation}
  We deduce from \eqref{eq:sum-x-lra-k} and \eqref{eq:yel-ersatz} that
  \begin{equation}
    \sum\limits_{k=1}^n \delta_{Y_{e, l}}(T_k) \le 4 \sum\limits_{l=0}^{n-1} \theta_l
    \label{eq:yel}
  \end{equation}
  and from \eqref{eq:sum-x-lra-k} and \eqref{eq:zxl-ersatz} that
  \begin{equation}
    \sum\limits_{k=1}^n \delta_{Z_{x, l}}(T_k) \le 4d \sum\limits_{l=0}^{n-1} \theta_l
    \label{eq:zel}
  \end{equation}
  Now, applying Theorem~\ref{thm:OSSS} on all $T_k$ and summing yields
  \begin{equation}
    n \Var(\ind_{A_n}) \le \sum\limits_{e, l} \Cov(Y_{e, l}, \ind_{A_n}) \sum\limits_{k=1}^n \delta_{Y_{e, l}}(T_k) + \sum\limits_{x, l} \Cov(Z_{x, l}, \ind_{A_n}) \sum\limits_{k=1}^n \delta_{Z_{x, l}}(T_k).
  \end{equation}
  Finally, plugging \eqref{eq:yel} and \eqref{eq:zel} gives the desired result.
\end{proof}

\subsection{Inequalities related to pivotal variables}

In this model, Theorem~\ref{thm:Russo} and the chain rule give
\begin{equation}
  \theta_n'(\beta) = \sum\limits_{l=1}^M p_l'(\beta) \sum\limits_{e \in E} \Proba(Y_{e, l} \pivotal A_n)
  \label{eq:Russo-fw}
\end{equation}
Moreover, for a base variable $X$, according to Lemma~\ref{lem:cov},
\begin{equation}
  \Cov(X, \ind_{A_n}) = \Var(X) \Proba(X \pivotal A_n) \le \frac 14 \Proba(X \pivotal A_n)
  \label{eq:pivot-cov}
\end{equation}
Therefore, in order to conclude, we need inequalities relating the probabilities that neighbouring variables are pivotal, not unlike the case of site and bond percolation.

\begin{lemma}
  Let $x \in \Z^d, 1 \le k \le N$ and $n \in \N^*$ as well as $\beta \in \R_+^*$. Let $l_0 \in \llbracket 1, M \rrbracket$ such that for all $s, t \in \llbracket 1, N \rrbracket$ we have $(\zeta(s, t) \ge 1) \wedge (s \ge k) \implies \zeta(s, t) \ge l_0$. Then
  \begin{equation}
    \Proba(Z_{x, k} \pivotal A_n) \le \frac {(2d-1)(1-p_{l_0})} {\Proba(Z_{x, k} = 1) (1 - h_M)^{2d}} \sum\limits_{y \in \nei(x)} \Proba(Y_{\{x, y\}, l_0} \pivotal A_n) .
  \end{equation}
  \label{lem:zxk-ineq}
\end{lemma}


\begin{remark}
  $l_0$ can be one if simplicity and generality are required (which is the case of our main theorem). However, with more knowledge on $\zeta$, the choice of $l_0$ provides some flexibility to achieve better bounds (see Remark~\ref{rem:better-bounds}). For example, one can choose $1 \le l_0 \le 1 \vee \zeta(k, 0)$ when it is relevant.
\end{remark}

\begin{remark}
  Unlike in Lemma~\ref{lem:ineq-pivot-site-bond}, we only have one inequality. However, we can prove an inequality in the other direction with similar methods under the additional assumption that $s \wedge t = 0 \implies \zeta(s, t) = 0$. Indeed, in that case, for an edge $e = \{x, y\}$ with $Y_{e, l} = 1$ pivotal, we can control the state of $e$ using the variable $Z_{x, 1}$ or $Z_{y, 1}$.
\end{remark}

\begin{proof}
  We observe that if $\omega \in \Omega$ makes $Z_{x, k}$ pivotal then the value of $\nu_x$ is different in $\omega_{(x, k)}$ and in $\omega^{(x, k)}$. Therefore, $\nu_x(\omega^{(x, k)}) \ge k$. Hence, for $y \in \nei(x)$, if $\zeta(\nu_x(\omega^{(x, k)}), \nu_y(\omega^{(x, k)})) \ge 1$ then $\zeta(\nu_x(\omega^{(x, k)}), \nu_y(\omega^{(x, k)})) \ge l_0$.\bigskip

  \emph{First,} we assume that $x \in \{0\} \cup \partial \Lambda_n$. Let $\omega \in \Omega$ where $Z_{x, k}$ is pivotal. The edges which are open in $\omega^{(x, k)}$ but not in $\omega_{(x, k)}$ have all $x$ for an endpoint, and at least one of them connects $x$ to a cluster intersecting $\{0\}$ (if $x \ne 0$) or $\partial \Lambda_n$ (if $x=0$). Let $e = \{x, y\}$ be such an edge.
  For $$\omega^{(x, k)}_{\{(\{x, z\}, l) : z \in \nei(x) \backslash \{y\}, l \in \llbracket 1, M \rrbracket\}}$$ we have that $e$ is open and pivotal. Indeed, we set $Z_{x, k}$ to $1$ so $e$ is open, and we closed all edges $\{x, z\}$ with $z \in \nei(x) \backslash \{y\}$. Hence, the aforementioned path from $0$ to $\partial \Lambda_n$ is unaffected. However, any path from $0$ to $\partial \Lambda_n$ has to have $x$ as an endpoint and therefore goes through $e$. Note that $e$ is open, so $\zeta(\nu_x, \nu_y) \ge 1$. Thus $\zeta(\nu_x, \nu_y) \ge l_0$.
\smallskip

  We set $I_y = \{(\{x, z\}, l) : z \in \nei(x), l \in \llbracket 1, M \rrbracket, (z \ne y) \vee (l \ne l_0)\}$. For $\omega^{(x, k)}_{I_y}$ (that is, we also change the values of $Y_{e, l}$ to $0$ whenever $l \ne l_0$), the edge $e$ might be closed but $Y_{e, l_0}$ is pivotal (remember that $\zeta(\nu_x, \nu_y) \ge l_0$). We can quantify these facts as follows. Let $y \in \nei(x)$ and $e = \{x, y\}$. We define $B_y$ the event that $Y_{e, l_0}$ is pivotal for $\omega^{(x, k)}_{I_y}$. It must be noted that $B_y, Z_{x, k}, (Y_{f, l})_{(f, l) \in I_y}$ are independent. Therefore,
  \begin{equation}
  \begin{aligned}
    \Proba(B_y) \Proba(Z_{x, k} = 1) & \, \frac {((1-p_{l_0}) \ldots (1-p_M))^{2d}} {1-p_{l_0}}
    \\ = & \Proba(B_y \cap (Z_{x, k} = 1, \forall (f, l) \in I_y, Y_{f, l} = 0)) \le \Proba(Y_{e, l_0} \pivotal A_n)
    \end{aligned}
  \end{equation}
  The previous statements amount to saying that
  \begin{equation}
    (Z_{x, k} \pivotal A_n) \subset \bigcup\limits_{y \in \nei(x)} B_y.
  \end{equation}
  Therefore,
  \begin{align}
    \Proba(Z_{x, k} \pivotal A_n) & \le \sum\limits_{y \in \nei(x)} \Proba(B_y) \\
                                  & \le \frac {1 - p_{l_0}} {\Proba(Z_{x, k} = 1) ((1-p_1) \ldots (1-p_M))^{2d}} \sum\limits_{y \in \nei(x)} \Proba(Y_{\{x, y\}, l_0} \pivotal A_n).
    \label{eq:zxk-pivot-boundary}
  \end{align}
\smallskip

  \emph{Second,} we proceed to the case $x \in \Lambda_{n-1} \backslash \{0\}$. Let $\omega \in \Omega$ where $Z_{x, k}$ is pivotal. Then $x$ is pivotal. In $\omega^{(x, k)}$ there are paths from $0$ to $\partial \Lambda_n$ and they all go through $x$ (instead of it being an endpoint) so there are two open edges $\{x, y_1\} \ne \{x, y_2\}$ such that there is a path from $0$ to $\partial \Lambda_n$ going through $y_1, x, y_2$.
For $$\omega^{(x, k)}_{\{(\{x, y\}, l) : y \in \nei(x) \backslash \{y_1, y_2\}, l \in \llbracket 1, M \rrbracket\}}$$ the edges $\{x, y_1\}$ and $\{x, y_2\}$ are open and pivotal. Indeed, the aforementioned path from $0 \lra y_1 \lra x \lra y_2 \lra \partial \Lambda_n$ is unaffected, and we closed all edges $\{x, y\}$ with $y \in \nei(x) \backslash \{y_1, y_2\}$ so any path from $0$ to $\partial \Lambda_n$ must go through $x$, therefore through $y_1$ and $y_2$. Note that $\{x, y_1\}$ is open so $\zeta(\nu_x, \nu_{y_1}) \ge 1$. It implies that $\zeta(\nu_x, \nu_{y_1}) \ge l_0$.\smallskip

  We set \smash{$J_{y_1, y_2} = \{(\{x, y\}, l) : y \in \nei(x) \backslash \{y_2\}, l \in \llbracket 1, M \rrbracket, (y \ne y_1) \vee (l \ne l_0)\}$}. For $\omega^{(x, k)}_{J_{y_1, y_2}}$ the variable $Y_{\{x, y_1\}, l}$ is pivotal. Indeed, $\{x, y_2\}$ is still open, $\{x, y_1\}$ pivotal, and the state of $\{x, y_1\}$ is controlled by $Y_{\{x, y_1\}, l_0}$ (recall that $\zeta(\nu_x, \nu_{y_1}) \ge l_0$).
  We denote by $B_{y_1, y_2}$ the event that $Y_{\{x, y_1\}, l_0}$ is pivotal for \smash{$\omega^{(x, k)}_{J_{_1, y_2}}$}. Since $B_{y_1, y_2}$ is independent of $Z_{x, k}$ and $Y_{e, l}$ for $(e, l) \in J_{y_1, y_2}$, we have
  \begin{align}
    \Proba(Z_{x, k} = 1) \frac {((1-p_1) \ldots (1-p_M))^{2d-1}} {1-p_{l_0}} \Proba(B_{y_1, y_2}) & = \Proba(B_{y_1, y_2}, Z_{x, k} = 1, \forall (e, l) \in J_{y_1, y_2}, Y_{e, l} = 0) \\
                                                                                             &\le \Proba(Y_{\{x, y_1\}, l_0} \pivotal A_n).
  \end{align}
  The discussion above implies that
  \begin{align}
    \Proba(Z_{x, k} \pivotal A_n) & \le \sum\limits_{y_1, y_2 \in \nei(x), y_1 \ne y_2} \Proba(B_{y_1, y_2}) \\
                                  & \le \frac {(2d-1)(1-p_{l_0})} {\Proba(Z_{x, k} = 1) ((1-p_1) \ldots (1-p_M))^{2d-1}} \sum\limits_{y \in \nei(x)} \Proba(Y_{\{x, y\}, l_0} \pivotal A_n).
    \label{eq:zxk-pivot-interior}
  \end{align}
  We now combine \eqref{eq:zxk-pivot-interior} and \eqref{eq:zxk-pivot-boundary} and notice that $(1-p_1) \ldots (1-p_M) = 1 - h_M$ to obtain the desired result.
\end{proof}

\subsection{Conclusion}

With these results in hand, we can confidently unwind the proof as for bond/site percolation.

\begin{theorem}
  We assume that there exists $\beta_c \in \R_+^*$ such that $ \theta(\beta) = 0$ for all $\beta \in ]0, \beta_c[$ and $ \theta(\beta) > 0$ for all $\beta \in ]\beta_c, +\infty[$. Then,
  \begin{itemize}
    \item for all $\beta \in ]0, \beta_c[$ there exists $c_\beta > 0$ such that for any $n$ large enough, $\theta_n(\beta) \le \exp(- c_\beta n)$;
    \item there exists $c, \delta > 0$ such that $ \theta(\beta) \ge c (\beta - \beta_c)$ for all $\beta \in [\beta_c, \beta_c + \delta[$.
  \end{itemize}
\end{theorem}

\begin{proof}
  Let $\beta_0 \in \R_+^*$. For $0 <\beta \le \beta_0$, we obtain:
  First, using \eqref{eq:Russo-fw},
  \begin{equation}
    \theta_n'(\beta) \ge \left( \min\limits_{1 \le l \le M} p_l'(\beta) \right) \sum\limits_{e \in E} \sum\limits_{l=1}^M \Proba(Y_{e, l} \pivotal A_n).
  \end{equation}
  Second, using Lemma~\ref{lem:zxk-ineq} with $l_0 = 1$ and denoting $u = \frac {(2d-1)(1-p_1)} {(1-h_M)^{2d} \min\limits_{1 \le k \le N} \Proba(Z_{0, k} = 1)}$,
  \begin{multline}
    \theta_n'(\beta) \ge \left( \min\limits_{1 \le l \le M} p_l'(\beta) \right) \Biggl( \frac 1 {1 + 2d N u(\beta)} \sum\limits_{e \in E} \sum\limits_{l=1}^M \Proba(Y_{e, l} \pivotal A_n) \\
    + \frac {2d N u(\beta)} {1 + 2d N u(\beta)} \frac 1 {2N u(\beta)} \sum\limits_{x \in \Z^d} \sum\limits_{k=1}^N \Proba(Z_{x, k} \pivotal A_n) \Biggr).
  \end{multline}
  Hence, because of \eqref{eq:pivot-cov},
  \begin{equation}
    \theta_n'(\beta) \ge \left( \min\limits_{1 \le l \le M} p_l'(\beta) \right) \frac 1 {1 + 2d N u(\beta)} \left( 4 \sum\limits_{e \in E} \sum\limits_{l=1}^M \Cov(Y_{e, l}, \ind_{A_n}) + 4d \sum\limits_{x \in \Z^d} \sum\limits_{k=1}^N \Cov(Z_{x, k}, \ind_{A_n}) \right).
  \end{equation}
  We can now use Lemma~\ref{lem:4.6-fw}:
  \begin{align}
    \theta_n'(\beta) & \ge \left( \min\limits_{1 \le l \le M} p_l'(\beta) \right) \frac 1 {1 + 2d N u(\beta)} \cdot \frac n {\sum\limits_{k=0}^{n-1} \theta_k} \theta_n (1 - \theta_n) \\
                     & \ge \left( \min\limits_{1 \le l \le M} p_l'(\beta) \right) \frac {1 - \theta_1(\beta_0)} {1 + 2d N u(\beta)} \cdot \frac n {\sum\limits_{k=0}^{n-1} \theta_k} \theta_n. \label{eq:final-ineq}
  \end{align}
  It allows us to apply Lemma~\ref{lem:3.2} on all segments of $\R_+^*$. Since there is a phase transition at $\beta_c$, we apply Lemma~\ref{lem:3.2} on a segment whose interior contains $\beta_c$.
\end{proof}

\begin{remark}
  There is a factor $N$ in the constant of \eqref{eq:final-ineq}. We may remove it if there is a constant $c > 0$ such that $\forall s, t \in \llbracket 1, N \rrbracket, \zeta(s, t) \ge 1 \implies \zeta(s, t) \ge cs$ (this is the case with the sum, the maximum and the product for example). Fix $x \in \Z^d$. For $k \in \llbracket 1, N \rrbracket$, we choose $l_k = \lceil ck \rceil$. It takes at most $\lceil c^{-1} \rceil$ times the same value. Therefore, when we apply Lemma~\ref{lem:zxk-ineq} and sum over $k$,
  \begin{align}
    \sum\limits_{k=1}^N \Proba(Z_{x, k} \pivotal A_n) & \le \sum\limits_{k=1}^N \frac {(2d-1)(1-p_{l_k})} {(1 - h_M)^{2d} \Proba(Z_{x, k} = 1)} \sum\limits_{y \in \nei(x)} \Proba(Y_{\{x, y\}, l_k} \pivotal A_n) \\
                                                      & \le \frac {(2d-1)} {(1 - h_M)^{2d} \min\limits_{k \in \llbracket 1, N \rrbracket} \Proba(Z_{x, k} = 1)} \lceil c^{-1} \rceil \sum\limits_{y \in \nei(x)} \sum\limits_{l=1}^M \Proba(Y_{\{x, y\}, l} \pivotal A_n).
  \end{align}
  If $c \ge 1$, we can do even better. Indeed, according to Lemma~\ref{lem:zxk-ineq},
  \begin{equation}
    (\lfloor c(k+1) \rfloor - \lfloor ck \rfloor)\Proba(Z_{x, k} \pivotal A_n) \le \frac {(2d-1)} {(1 - h_M)^{2d} \Proba(Z_{x, k} = 1)} \sum\limits_{y \in \nei(x)} \sum\limits_{l=\lfloor ck \rfloor}^{\lfloor c(k+1) \rfloor - 1} \Proba(Y_{\{x, y\}, l} \pivotal A_n).
  \end{equation}
  Note that $\lfloor c(k+1) \rfloor - \lfloor ck \rfloor \ge \lfloor c \rfloor$. Therefore, when summing over $k$,
  \begin{equation}
      \sum\limits_{k=1}^N \Proba(Z_{x, k} \pivotal A_n) \le \frac {(2d-1)} {(1 - h_M)^{2d} \min\limits_{k \in \llbracket 1, N \rrbracket} \Proba(Z_{x, k} = 1)} \frac 1 {\lfloor c \rfloor} \sum\limits_{y \in \nei(x)} \sum\limits_{l=1}^M \Proba(Y_{\{x, y\}, l} \pivotal A_n).
  \end{equation}
  We traded the factor $N$ for $\lceil c^{-1} \rceil$ or $\frac 1 {\lfloor c \rfloor}$ depending on the case. However, the constant still depends heavily on $N$, making it difficult to extend \eqref{eq:final-ineq} (and therefore sharp phase transition) to other distributions $q_\nu$ by taking a limit as $N$ goes to $+\infty$.
  \label{rem:better-bounds}
\end{remark}

{\bf Acknowledgment:} This note was written while the first author was on a research internship at the University of Cologne. The authors would like to thank Alexander Drewitz for useful discussions.